\documentclass[11pt]{amsart}
\usepackage{amsmath,amssymb, amscd}
\usepackage{graphicx}

\usepackage{verbatim}

\def\sqr#1#2{{\vcenter{\hrule height.#2pt
        \hbox{\vrule width.#2pt height#1pt \kern#1pt
                \vrule width.#2pt}
        \hrule height.#2pt}}}

\newtheorem{theorem}{Theorem}[section]
\newtheorem{lemma}[theorem]{Lemma}
\newtheorem{proposition}[theorem]{Proposition}
\newtheorem{corollary}[theorem]{Corollary}

\theoremstyle{definition}
\newtheorem{definition}[theorem]{Definition}
\newtheorem{example}[theorem]{Example}

\newtheorem{remark}[theorem]{Remark}
\newtheorem{setting}[theorem]{Setting}

\newtheorem{claim}[theorem]{Claim}

\newtheorem{discussion}[theorem]{Discussion}
\newtheorem{construction}[theorem]{Construction}

\DeclareMathOperator{\n}{\mathfrak n}
\DeclareMathOperator{\m}{\mathfrak m}
\DeclareMathOperator{\p}{\mathfrak p}
\DeclareMathOperator{\q}{\mathfrak q}

\DeclareMathOperator\ord{ord}

\DeclareMathOperator{\hgt}{ht}

\DeclareMathOperator{\rat.rank}{rat.rank}
\DeclareMathOperator{\tr.deg}{tr.deg}

\DeclareMathOperator{\N}{\mathbb N}
\DeclareMathOperator{\Z}{\mathbb Z}
\DeclareMathOperator{\Q}{\mathbb Q}

\DeclareMathOperator{\Bl}{Bl}

\def\alert#1{\smallskip{\hskip\parindent\vrule%
\vbox{\advance\hsize-2\parindent\hrule\smallskip\parindent.4\parindent%
\narrower\noindent#1\smallskip\hrule}\vrule\hfill}\smallskip}

\begin{document}

\title[directed unions of local quadratic transforms ]
{Directed unions of local quadratic transforms \\ of a regular local ring }


\author{William Heinzer}
\address{Department of Mathematics, Purdue University, West
Lafayette, Indiana 47907 U.S.A.}
\email{heinzer@math.purdue.edu}

\author{Mee-Kyoung Kim$^{1}$}
\address{Department of Mathematics, Sungkyunkwan University, Jangangu Suwon
440-746, Korea} 
\email{mkkim@skku.edu}

\thanks{$^{1}$ This paper was supported by Faculty Research Fund, 
Sungkyunkwan University, 2013.}

\author{Matthew Toeniskoetter}
\address{Department of Mathematics, Purdue University, West
Lafayette, Indiana 47907 U.S.A.}
\email{mtoenisk@math.purdue.edu}

\thanks{Correspondence with Alan Loper, Bruce Olberding
and Hans Schoutens  that motivated our interest in  the work of Shannon and Granja on
infinite directed unions of 
local quadratic transformations is gratefully acknowledged. }

\date \today


\begin{abstract} 
	Let $(R,\m)$ be a $d$-dimensional regular local domain with $d \ge 2$ and 
	let $V$ be a valuation domain  birationally dominating  $R$ 
	such  that the residue field of $V$ is algebraic over $R/\m$.  
	Let $v$ be a valuation associated to $V$. Associated to $R$ and $V$ there exists an
	infinite directed family  $\{(R_n, \m_n)\}_{n \ge 0}$   of $d$-dimensional regular local rings dominated by $V$ 
	with  $R = R_0$ and    $R_{n+1}$  the  local
	quadratic transform of $R_n$ along $V$.    
	Let 
	$S := \bigcup_{n \ge 0}R_n$.    Abhyankar proves  that $S = V$ if $d = 2$.   Shannon 
	observes that often $S$ is properly contained in $V$ if $d \ge 3$, and Granja
	gives necessary and sufficient conditions for $S$ to be equal to $V$.    The directed 
	family $\{(R_n,\m_n)\}_{n \ge 0}$  and the integral domain $S = \bigcup_{n \ge 0}R_n$ may be
	defined without first prescribing a dominating valuation domain $V$.    
	If $\{(R_n, \m_n)\}_{n \ge 0}$ switches strongly infinitely often, then 
	$S = V$ is a rank one valuation domain and for nonzero elements $f$ and $g$ in $\m$, 
	we have  $ \frac{v (f)}{v (g)} ~= ~ \underset{n \to \infty}\lim  \frac{\ord_{R_n} (f)}{\ord_{R_n} (g)}$.  
	If $\{(R_n, \m_n)\}_{n \ge 0}$ is a 
	family of monomial local quadratic transforms, we give  necessary and sufficient conditions for $\{(R_n, \m_n)\}_{n \ge 0}$ to 
	switch strongly infinitely often. If these conditions hold, then $S = V$ is a rank one valuation domain of 
	rational rank $d$ and $v$ is a monomial valuation.   
	Assume that $V$ is rank one and birationally dominates $S$. Let 
	$s = \sum_{i=0}^{\infty} v (\m_i)$. Granja, Martinez and Rodriguez show that $s = \infty$ implies $S = V$.
	We prove that $s$ is finite if $V$ has rational rank at least $2$.
	In the case where $V$ has maximal rational rank, we give a sharp upper bound for $s$ and 
	show that $s$ attains this bound if and only if the sequence switches strongly infinitely often.
\end{abstract}

\maketitle
\bigskip

\baselineskip 18 pt

\section{Introduction} \label{c1}
Let $R = R_0$ be a $d$-dimensional regular local ring,  and for each integer $n \ge 0$,
let $R_{n+1}$ be a $d$-dimensional local quadratic transform of $R_n$.  Thus  
$\{(R_n,\m_n)\}_{n \ge 0}$ is  a directed family of 
$d$-dimensional  regular local rings.   Let 
$S := \bigcup_{n \ge 0}R_n$.   If $d = 2$,      
Abhyankar proves in \cite{A} that $S$ is always a valuation domain.    
In the case where  $d \ge 3$,  
Shannon in \cite{S} and later Granja in \cite{Gr} 
consider conditions in order that $S$ be a valuation domain. In this
connection,  Shannon gives  the following definition in  \cite[page~314]{S}.

\begin{definition}  Let $\{(R_n, \m_n)\}_{n \ge 0}$ be an infinite directed family 
of local quadratic transforms of a regular local ring  $(R, \m)$.   
We say that  $\{(R_n, \m_n)\}_{n \ge 0}$  {\bf  switches strongly infinitely often} if there
 does not exist an integer $j$ and a height one prime ideal  $\p_j$ of $R_j$ 
with the property that $\bigcup_{n=0}^{\infty} R_n  \subset (R_j)_{\p_j}$.
\end{definition}

Assume that $V$ is a rank one valuation domain that birationally dominates $S:= \bigcup_{n\geq 0}R_n$.
If $V$ is non-discrete, 
Shannon proves in \cite[Proposition~4.18]{S} that $S=V$ 
if and only if  $\{(R_n, \m_n)\}_{n \ge 0}$ switches strongly infinitely often. 
It is observed in \cite[Theorem~6]{GR} that the proof given by Shannon also holds
if $V$ is rank one discrete. 

Granja \cite[Proposition~7]{Gr} shows that if $S := \bigcup_{n\geq 0}R_n$ is a valuation domain $V$,
then $V$ has real rank either one or two, and in \cite[Theorem~13]{Gr}, he characterizes the sequence of local quadratic transforms of $R$ along $V$. 
If $V$ has rank one, then $\{(R_n, \m_n)\}_{n \ge 0}$ switches strongly infinitely often.
If $V$ has rank two, then the value group of $V$ is 
$\Z \oplus ~ G$, where $G$ has rational rank one. In this
case  Granja proves \cite[Theorem~13]{Gr} that
 the sequence $\{(R_n, \m_n)\}_{n \ge 0}$ is height one directed as in Definition~\ref{htonedir}.

\begin{definition} \label{htonedir}
  Let $\{(R_n, \m_n)\}_{n \ge 0}$ be an infinite directed family 
of local quadratic transforms of a regular local ring  $(R, \m)$.   
The sequence  $\{(R_n, \m_n)\}_{n \ge 0}$ is  {\bf  height one directed} if there
exists a nonnegative integer $j$ and a height one prime ideal $\p$ of $R_j$
such that    $\bigcup_{n=0}^{\infty} R_n  \subset (R_j)_{\p}$, and if
for some nonnegative integer $k$ and some height one prime ideal $\q$ of $R_k$ we have
$\bigcup_{n=0}^{\infty} R_n  \subset (R_k)_{\q}$, then $(R_j)_{\p} = (R_k)_{\q}$.
\end{definition}

Let $\{ (R_n, \m_n) \}_{n \ge 0}$ be a directed family of Noetherian local domains.
Assume that $\ord_{R_n}$ defines a valuation for each $n$.
If \,$\bigcup_{n=0}^{\infty} R_n = V$ is a rank one valuation domain,
we prove in Theorem~\ref{approx}   that the valuation $v$ is  related to  the order valuations of the $R_n$
as follows:  
	$$\frac{v (f)}{v (g)} ~ =   ~  \lim_{n \rightarrow \infty} \frac{\ord_{R_n} (f) }{\ord_{R_n} (g)},$$
for all nonzero elements $f,g$ in the maximal ideal of $V$.

 Let $(R,\m)$ be a $d$-dimensional regular local ring with $d \ge 2$ and fix $d$ elements $x, y, \ldots, z$ such that $\m = (x,y, \ldots, z)R$.  
In Section~\ref{c3}, we consider finite sequences $R = R_0 \subset R_1 \subset \cdots \subset R_n$  of monomial local quadratic transforms.
We prove in  Theorem~\ref{3.3a}  that the order in $R$ of each nonzero $f \in \m$   is greater than the order of the transform of $f$ in $R_n$ 
if and only if in passing from $R_0$ to $R_n$ we  transform in each monomial direction at least once.

In Section~\ref{c4}, we consider infinite sequences 
$\{(R_n,\m_n)\}_{n \ge 0}$ of monomial  local quadratic transforms.
In Theorem~\ref{3.3}, we prove that $\bigcup_{n=0}^{\infty} R_n$ is a rank one valuation
domain if and only if the transform in each monomial direction occurs infinitely many times. 

In Section~\ref{c5}, we present examples  of infinite sequences of local quadratic transforms 
$\{(R_n, \m_n)\}_{n \ge 0}$ 
that switch strongly infinitely often and have the property that $\bigcup_{\n \ge 0}R_n = V$ 
is 
a rank one valuation ring having rational rank less than $\dim R$.

Let $\{ (R_n, \m_n) \}_{n \ge 0}$ be a directed sequence of local quadratic transforms along a zero-dimensional rank one valuation $V$.
In Section~\ref{c6}, we consider the invariant $s = \sum_{i=0}^{\infty} v (\m_i)$.
Granja, Martinez and Rodriguez prove in \cite[Prop. 23]{GMR} that
 $s = \infty$ implies $\bigcup_{n \ge 0}R_n  = V$.
We observe in Proposition~\ref{6.2} that $s$ is finite if $V$ has rational rank at least $2$ and that 
$s = \infty$ or $s < \infty$ are  both possible if $V$ has rational rank one and is not a DVR.
In the case where $V$ has maximal rational rank, we give in Theorem~\ref{6.3} a sharp upper 
bound for $s$,  and show that $s$ attains this bound if and only if the sequence 
switches strongly infinitely often. In Theorem~\ref{htonedirected}, we give necessary and 
sufficient conditions for a sequence $\{ (R_n, \m_n) \}_{n \ge 0}$ to be 
height one directed. This yields  examples where $S = \bigcup_{n \ge 0}R_n$ is a rank 2 valuation
domain with value group $\mathbb Z \oplus H$ such that  $H$ is rational rank one but not discrete.

We use $\mu (I)$ to denote the minimal number of generators of an ideal $I$,
and $\lambda_R(M)$ to denote the length of an $R$-module $M$. We use the 
notation $A \subset B$ to denote that $A$ is a subset of $B$ that may be equal to $B$.

\section{Preliminaries} \label{c2}

\begin{definition} \label{1.1}
Let $R \subseteq T$ be unique factorization domains (UFDs)  with  $R$
and $T$ having the same field of fractions.  Let $p$ be a height-one prime
in $R$. If $pT \cap R = p$, then there exists a unique height-one prime $q$ 
in $T$ such that $q \cap R = p$. We then have $R_p = T_q$. On the other hand,
if $p \subsetneq pT \cap R$, then $(R \setminus p)^{-1}T = \mathcal Q(T)$,
the field of fractions of $T$. The {\bf transform } $p^T$ of $p$ in $T$ is
the ideal $q$ if $pT \cap R = p$ and is the ring $T$ if $p \subsetneq pT \cap R$.
Thus
\begin{equation*}
p^{T} ~ = ~ \begin{cases} q \quad &\text{ if }\quad p=pT \cap R, \\
                           T \quad &\text{ if }\quad p \subsetneq pT \cap R.
\end{cases}
\end{equation*}
Let $I$ be a nonzero ideal of $R$. 
Then $I$ has a unique factorization $I = p_1^{a_1} \cdots p_n^{a_n}J$, where the
$p_i$ are principal prime ideals, the $a_i$ are positive integers, and 
$J$ is an ideal with $J^{-1} = R$.  
\begin{enumerate}
\item The  {\bf transform } $I^T$ of $I$ in $T$ is the ideal 
$$ 
I^T = q_1^{a_1} \cdots q_n^{a_n}(JT)(JT)^{-1}
$$
where $q_i = p_i^T$ for each $i$ with $1 \le i \le n$. 
\item  
The {\bf complete transform} of $I$ in $T$ is the completion $\overline {I^T}$ of $I^T$.
\end{enumerate} 
\end{definition} 

In the case where $J$ is a nonzero principal ideal, the definition here of the transform agrees with the definition given by Granja 
\cite[page~701]{Gr} for the strict transform.

For an ideal $I$ of a local ring $(R,\m)$, the {\bf order} of $I$, denoted $\ord_{R} I$, is $r$ 
if $I \subseteq \m^r$ but $I \nsubseteq \m^{r+1}$.
 If $(R,\m)$ is a regular
local ring, the function that associates to an element $a \in R$,   
the order of the principal ideal $aR$,   defines  a 
discrete rank-one valuation, denoted $\ord_R$  on the field of fractions of $R$.  The associated valuation 
ring (DVR)  is called {\bf the order valuation ring } 
of $R$.

\begin{definition}\label{2.02}
Let $V$ be a valuation domain corresponding to the valuation $v$. The {\bf{rank}} of $v$
is defined to be the Krull dimension of $V$. The {\bf{rational rank}} of $v$, denoted $\rat.rank v$, 
is the rank of the value group $\Gamma_v$ of $V$ over $\mathbb Q$.  Thus,
$\rat.rank v~=~\dim_{\mathbb Q}(\Gamma_{v} \otimes_{\mathbb Z}\mathbb Q)$.
\end{definition}

\begin{definition}\label{2.03}
Let $R$ be an integral domain.  An ideal $I$ of $R$ is said to be a
{\bf{valuation ideal}} if there exists a valuation domain $V$ such that 
$R \subseteq V \subseteq \mathcal Q(R)$ and $IV \cap R~=~I$.
We then say that $I$ is a $V$-{\bf ideal} in $R$,  and if $v$ is a valuation corresponding to $V$ 
that $I$ is a  {\bf{$v$-ideal}} in $R$.
\end{definition}

If $R$ is a subring of a valuation domain  $V$   and  $\m_V$   is the  maximal ideal of $V$,  then 
the prime ideal $\m_V \cap R$ of $R$   is  called  {\bf the center} of $V$ on $R$.
Let $(R,\m)$ be a Noetherian local domain with field of fractions $\mathcal Q(R)$.
A valuation domain  $(V, \m_V)$ is said to {\bf birationally dominate}  $R$ 
if $R \subseteq V \subseteq \mathcal Q(R)$ and  $\m_V  \cap R = \m$, that is, $\m$ is the 
center of $V$ on $R$. We write  $R ~ \overset{b.d.}\subset ~ V$ to denote 
that $V$ birationally dominates $R$.
The valuation domain  $V$ is said
to be a {\bf prime divisor} of $R$ if  $V$ birationally dominates $R$ 
and the transcendence degree of the field $V/\m_V$ 
over $R/\m$ is $\dim R - 1$. If $V$ is a prime divisor of $R$, then $V$ is a DVR
\cite[p. 330]{A}.

\begin{definition}\label{2.04}
Let $R$ be a local domain with maximal ideal $\m$ and 
let  $V$ be a  valuation domain  that  birationally dominates  $R$. 
Let $v$ be the valuation associated with $V$.
The {\bf{dimension}} of $v$ on $R$, or the {\bf dimension} of $V$ over $R$,
  is defined to be the transcendence degree of the residue field $k(v)$ of $v$ over  
the field  $R/\m$.
\end{definition}

The {\bf quadratic dilatation}  or {\bf blowup}  of $\m$ along $V$,  
cf.  \cite[ page~141]{N},  is the
unique local ring on the blowup $\Bl_{\m}(R)$  of $\m$ that is dominated by  $V$.  
The ideal $\m V$ is 
principal and is generated by an element of $\m$.  Let $a \in \m$ be such that $aV = \m V$.  Then
$R[\m/a] \subset V$. Let $Q := \m_V  \cap R[\m/a]$. Then $R[\m/a]_Q$ is the  { \bf  quadratic 
transformation   of }
$R$ {\bf along }  $V$.  In the special case where $(R,\m)$ is a $d$-dimensional regular local domain 
we use the following terminology.

\begin{definition}\label{2.1}
Let $d$ be a positive integer and let $(R, \m, k)$ be a $d$-dimensional regular local ring 
with maximal ideal $\m$ and  residue field $k$. 
Let $x\in \m \setminus \m^2$ and let $S_1 := R[\frac{\m}{x}]$. The ring  $S_1$ is a
$d$-dimensional regular ring in the sense that each localization of $S_1$ at a 
prime ideal is a regular local ring.  To see this, observe that  $S_1/xS_1$ is isomorphic to a 
polynomial ring in $d-1$ variables over the field $k$, cf. \cite[Corollary~5.5.9]{SH},
 and $S_1[1/x] = R[1/x]$ is a regular ring. Moreover,
$S_1$ is a UFD  since $x$ is a prime element of $S_1$ and
$S_1[1/x] = R[1/x]$ is a UFD, cf. \cite[Theorem~20.2]{M}.
Let $I$  be  an $\m$-primary ideal of $R$ with $r:=\ord_{R}(I)$. Then one has in $S_1$
$$
IS_1=x^rI_1 \quad \text{for some ideal}\quad I_1 \quad \text{of}\quad S_1.
$$
It follows  that either $I_1 = S_1$ or $\hgt I_1 \ge 2$. 
Thus  $I_1$ is  the transform $I^{S_1}$  of $I$ in $S_1$ as
in Definition~\ref{1.1}. 

Let $\p$ be a prime ideal of $R[\frac{\m}{x}]$ with $\m \subseteq \p$. 
      The local ring 
$$R_1:~= ~R[\frac{\m}{x}]_{\p}  ~ = (S_1)_{\p}
$$
 is called a {\bf local quadratic transform} of $R$;  the ideal
$I_1R_1$ is the  transform of $I$ in $R_1$ as in Definition~\ref{1.1}.  
\end{definition}

\begin{remark} \label{eqloc} Let $(R,\m)$ be a regular local ring and let
$(R_1, \m_1)$ be a local quadratic transform of $R$ as in Definition~\ref{2.1}. 
If $\q$ is a nonzero prime ideal of $R_1$ such that $\q \cap R =: \p$ is properly
contained in $\m$, then $R_{\p} = (R_1)_{\q}$. 
\end{remark} 
\begin{proof}  We have  $R_{\p} \subseteq (R_1){\q}$, and $(R_1){\q}$
dominates $R_{\p}$. 
Since $R_1$ is a localization of $S_1 = R[\frac{\m}{x}]$,
we have $xR_1 \cap R = \m$. Hence $\q \cap R \ne \m$ implies that $x \notin \p$.
Thus $R[\frac{1}{x}] \subset R_{\p}$, and so $R \subset S_1  \subset R_{\p} \subseteq (R_1){\q}$.
Since $(R_1)_{\q}$ is a localization of $S_1$ and $\q(R_1)_{\q} \cap S_1 = \p R_{\p} \cap S_1$,
we have $R_{\p} = (R_1)_{\q}$.
\end{proof}

\begin{definition}\label{changedir}
Let $(R, \m)$ be a $d$-dimensional regular local ring and consider a sequence of local quadratic transforms along $R$,
	$$R = R_0 \subset R_1 \subset \ldots \subset R_n.$$
We say {\bf there is a  change of direction }  from $R_0$ to $R_n$  if $\m_0 \subset \m_n^2$.
\end{definition}

\begin{remark}\label{6.4}
Assume the notation of Definition~\ref{changedir}.
Let $V$ be a valuation domain birationally dominating $R_n$ and let $v$ be a valuation
associated to $V$. 
There is a change of direction from $R_0$ to $R_n$ if and only if $v (\m_0) > v (\m_{n-1})$.
\end{remark}

\begin{proof}
It suffices to prove the case where  $n = 2$.
Assume that $v (\m_0) = v (\m_1)$. We show there is no change of direction from $R_0$ to $R_2$. 
Let $x \in \m_0$ be such that $v (x) = v (\m_0)$.
Then $R_2$ is a localization of $R_1 [ \frac{\m_1}{x} ]$.  
Hence $x \in \m_2 \setminus \m_2^2$.

Conversely, assume that $v (\m_0) > v (\m_1)$. We show there is a 
change of direction from $R_0$ to $R_2$.  Let $x \in \m_0$. Let 
$w \in \m_1$ be such that $v (w) = v (\m_1)$. 
Then  $R_2$ is a localization of $R_1 [ \frac{\m_1}{w} ]$.
Since $v (x) \ge v (\m_0) > v (\m_1) = v (w)$, 
we have $\frac{x}{w} \in \m_2$.  Thus  $x = w \frac{x}{w} \in \m_2^2$.
We conclude that $\m_0 \subset \m_2^2$. 
\end{proof}

\section{Directed unions of local quadratic transforms}

\begin{remark} \label{3.25} 
If $(R,\m)$ is a 2-dimensional regular local ring, then a well-known result 
of Abhyankar \cite[Lemma12]{A} states that an infinite directed union of 
local quadratic transforms of $R$ is always a valuation domain. Examples 
given by Shannon in \cite[Examples~4.7 and 4.17]{S} show that this fails
in general if $d \ge 3$. 
\end{remark}

\begin{remark}  \label{switch}  
Let $\{( R_n, \m_n) \}_{n \ge 0}$ be an infinite directed family of local quadratic transforms of a regular local ring $R$. 
As shown by Granja in \cite[Lemma~10]{Gr},  the following are equivalent:
\begin{enumerate}
\item The sequence $\{ (R_n, \m_n) \}_{n \ge 0}$ switches strongly infinitely often.
\item For each  integer $j \ge 0$ and nonzero  element $f \in R_j$, there exists a positive integer  $n \ge j$ 
such that  the transform $(f R_j)^{R_n} = R_n$.
\end{enumerate}
\end{remark}

Let $V$ be a valuation domain  birationally dominating a regular local ring
$(R,\m)$ and let $\{(R_n,\m_n)\}_{n \ge 0}$ be the family of local quadratic transforms of $R$ along $V$.
Granja    in \cite[Theorem~13]{Gr} proves the following:  

\begin{remark}  \label{rank}  Let $\{(R_n, \m_n)\}_{n \ge 0}$ be an infinite directed family of local quadratic transforms of a regular local ring  $R$.
If  $\{ R_n \}_{n \ge 0}$  switches strongly infinitely often, 
then $S := \bigcup_{n=0}^{\infty} R_n$ is a rank one  valuation domain.
\end{remark}

\begin{proof}
To prove that $S$ is a valuation domain, it suffices to show for nonzero  
elements $f, g \in R := R_0$  
that either $f / g$ or $g / f$ is in $R_{n}$ for some $n > 0$.
Let $I = (f, g) R$.
We prove the assertion  by induction on $\ord_{R} (I)$.
The case where $\ord_{R} (I) = 0$ is clear.
Let   $\ord_{R} (I) = r > 0$,   and assume the assertion holds for nonnegative 
integers less than $r$.  
We may assume that $r = \ord_{R} (f) \le \ord_{R} (g)$.
Let  $I^{R_1} = (f_1, g_1)R_1$, where $\frac{f_1}{g_1} = \frac{f}{g}$, denote 
the transform of $I$ in $R_1$. 
Since $r = \ord_{R} (f) = \ord_{R}(I)$, the ideal  $f_1R_1$  is the transform of $fR$ in $R_1$. 
If $\ord_{R_1} (f_1) < r$ or $\ord_{R_1} (g_1) < r$, then we are done by 
the induction hypothesis.
Otherwise, $\ord_{R_1} (I_1) = \ord_{R_1} (f_1) = r$ and $I_1 := (f_1,g_1)R_1$ 
satisfies the same hypotheses as $I$.
Let $I^{R_2} = (f_2,g_2)R_2$,  where $\frac{f_2}{g_2} = \frac{f}{g}$, 
denote the transform of $I$ in $R_2$, and
continue in this way to define $I_n = (f_n, g_n)R_n$.   
Since the sequence switches strongly infinitely often, there  exists a 
positive integer $n$ such that  either $\ord_{R_n}(g_n) < r$ or
the strict transform $f_n$ of $f$ in $R_n$ 
has $\ord_{R_n} (f_n) < r$.   This proves that 
$S = \bigcup_{n=0}^{\infty} R_n$ is a valuation domain.

Assume that  $\dim S > 1$ and let $P$ be a nonzero nonmaximal prime ideal  of $S$.  
For each $n \ge 0$, let $P_n := P \cap R_n$.   Then  $P = \bigcup_{n \ge 0}P_n$
and $S_P = \bigcup_{n \ge 0} (R_n)_{P_n}$.   
There exists a positive integer $t$ for which $P_t$ is a nonzero, 
nonmaximal prime ideal of the regular local ring $R_t$.  
For each positive integer $i$,  we have  
$P_{t+i} \cap R_t = P_t$.  Therefore by Remark~\ref{eqloc}  
 $(R_t)_{P_t} = (R_{t+i})_{P_{t+i}}$.  It follows that 
$S_P = (R_t)_{P_t}$,  a contradiction to our hypothesis that $\{R_n\}_{n \ge 0}$ switches strongly infinitely often.  
 We conclude
that $\dim S = 1$ and $S$ is a rank one valuation domain.
\end{proof}

We use the following setting to observe in  Theorem~\ref{valideals} 
a connection between  valuation ideals 
and  an infinite directed family that switches strongly infinitely often. 

\begin{setting} \label{validealset} 
Let $\{(R_n, \m_n)\}_{n \ge 0}$ be an infinite directed family of local quadratic transforms of a regular local ring  $R$.  
Assume that  $\{(R_n, \m_n)\}_{n \ge 0}$  switches strongly infinitely often and 
let $V = \bigcup_{n \ge 0}R_n$. As noted in Remark~\ref{rank}, $V$ is a rank one 
valuation domain. Let $v$ be a valuation associated to $V$. By  \cite[Lemma~3, p. 343]{ZS2}, 
the $V$-ideals of $R$ form a descending sequence $\{ I_n \}_{n \ge 0}$,  and $\bigcap_{n \ge 0}I_n = (0)$.
 \end{setting}

\begin{theorem} \label{valideals}
Assume notation as in Setting~\ref{validealset}.  
Then there exists  an ascending sequence of integers $\{ \tau_n \}_{n \ge 0}$ 
such that for each valuation domain $W$ that  birationally dominates $R_{\tau_n}$, 
the ideals $I_1, \ldots, I_n$ are valuation ideals for $W$.
\end{theorem}

\begin{proof}
We  construct the sequence $\{ \tau_n \}_{n \ge 0}$ by induction on $n$.
Suppose we have constructed the sequence up to length $n$.
Since the sequence of local quadratic transforms switches strongly infinitely often, 
there exists an integer $\tau_{n+1} \ge \tau_n$ such that $I_{n} R_{\tau_{n+1}}$ is principal,
 say $I_{n} R_{\tau_{n+1}} = g R_{\tau_{n+1}}$, where $g \in I_n$.
Since $V$ birationally dominates $R_{\tau_{n+1}}$, it is clear that $v (g) = v (I_n)$.
Notice that we may take $g$ to be any element in $I_n$ such that $v (g) = v (I_n)$.
Let $h$ be another such element.
Since $h \in g R_{\tau_{n+1}}$, $\frac{h}{g} \in R_{\tau_{n+1}}$.
Since $V$ birationally dominates $R_{\tau_{n+1}}$ and $v (g) = v (h)$, the element
$\frac{h}{g}$ is a unit in $R_{\tau_{n+1}}$, so $g R_{\tau_{n+1}} = h R_{\tau_{n+1}}$.

Let $W$ birationally dominate $R_{\tau_{n+1}}$. 
Then $W$ birationally dominates $R_{\tau_n}$.
By the induction hypothesis, $I_0, \ldots, I_n$ are $W$-ideals.
To see that $I_{n+1}$ is a $W$-ideal, consider $J = I_{n+1} W \cap R_0$.
Since $I_{n+1} \subsetneq I_n$ and $I_n$ is a $W$-ideal, 
we have $I_{n+1} \subset  J \subset  I_n$.
Suppose by way of contradiction that $J \ne I_{n+1}$.
Let $f \in J \setminus I_{n+1}$.
Since $I_{n+1} = \{ a \in R_0 \mid v (a) > v (I_n) \}$, it follows that  $v (f) \le v (I_n)$.
Because $f \in I_n$, we have  $v (f) = v (I_n)$.
Thus $I_n R_{\tau_{n+1}} = f R_{\tau_{n+1}}$, so $J = I_n$.
Since $I_{n+1} \subsetneq I_n$ and $I_n R_{\tau_{n+1}} = f R_{\tau_{n+1}}$, 
there exists a  proper ideal $L$ of $R_{\tau_{n+1}}$ such that $I_{n+1} R_{\tau_{n+1}} = f L$.
Since $L$ is a proper ideal of $R_{\tau_{n+1}}$ and $W$ dominates $R_{\tau_{n+1}}$, 
we have $w (L) > 0$.
Since $J = f L W \cap R$ and $f \in J$, it follows that 
 $f \in f L W$, which contradicts the fact that $w (L) > 0$.
\end{proof}

\begin{remark}\label{validealrem}
Let $(R, \m)$ be a Noetherian local domain and let $V$ be a valuation domain birationally 
dominating $R$ 
such that the vector space dimension  of $V/\m_V$ over $R/\m$ denoted 
$\dim_{R / \m} V / \m_V = e < \infty$.
Let $\{ I_n \}_{n \ge 0}$ be the sequence of $V$-ideals of $R$.
Then $1 \le \lambda_{R} (I_n / I_{n+1}) \le e$ for all $n \ge 0$.
To see this, fix an integer $n$, and consider the natural embedding as $R$-modules,
$$
\frac{I_n}{I_{n+1}} ~ \cong ~\frac{I_n V \cap R}{\m_V I_n V \cap R} ~ \hookrightarrow ~\frac{I_n V}{\m_V I_n V} ~\cong ~ \frac{V}{ \m_V},
$$
where the last isomorphism follows because $I_n V = fV$, for some element $f \in I_n$. 
Thus there is an $R$-module embedding of $I_n / I_{n+1}$ into $V / \m_V$, and hence we have 
 $0 < \lambda_{R} (I_n / I_{n+1}) \le \lambda_{R} (V / \m_V) = e$.
In particular,  if $R / \m = R_n / \m_n$ for all $n \ge 0$ and $\{ I_n \}_{n \ge 0}$ is
the sequence of $V$-ideals in $R$, then  $\lambda_{R} (I_n / I_{n+1}) = 1$ for all $n \ge 0$.
\end{remark}

\begin{remark}
Let $R$ and the family $\{(R_n,\m_n)\}_{n \ge 0}$ be as in Theorem~\ref{valideals}. It is natural to 
ask about conditions on the integer $\tau_n$ given in Theorem~\ref{valideals}. The smaller
the integer $\tau_n$, the sharper the assertion. We describe the situation in more detail 
in Discussion~\ref{discuss1} 
\end{remark}

\begin{discussion} \label{discuss1}
Let the ring  $R$, the family of local quadratic transforms  $\{(R_n,\m_n)\}_{n \ge 0}$,
and the descending 
chain of ideals $\{I_n\}_{n \ge 0}$  be as in Theorem~\ref{valideals}. Let $\mathcal S$ denote the
set of valuation domains that birationally dominate $R$.
For each integer $n \ge 0$,
let
$$
\begin{aligned}
\mathcal S_n  ~ &= ~ \{ W \in \mathcal S ~|~  
I_0, I_1, \ldots,  I_n  \quad \text{ are all $W$-ideals} \}, \\
\mathcal T_n  ~ &= ~ \{ W \in \mathcal S ~|~  
R_{n} ~ \overset{b.d.}\subset ~ W  \}.
\end{aligned}
$$
Notice that $\mathcal S_0 = \mathcal S_1 = \mathcal S = \mathcal T_0$ and the families  
$\{\mathcal S_n \}$ and $\{\mathcal T_n \}$ are both linearly ordered with respect to
inclusion and decreasing. Since $I_2 = P_{RR_1}$ is the special $*$-simple complete 
ideal in $R$ associated to $R_1$ as defined by Lipman in \cite[Proposition~2.1]{L}, 
we also have $\mathcal S_2 = \mathcal T_1$. 
Associated with these sequences, we define:
$$
\begin{aligned}
s_n  ~ &= ~ \min \{~ j \in \N_0 ~|~  \mathcal T_j ~\subseteq \mathcal S_n \}. \\
t_n  ~ &= \begin{cases} ~ \min \{~ j \in \N_0 ~|~  \mathcal S_j ~\subseteq \mathcal T_n \},\\
                ~ \infty ~\text{ if $\mathcal S_j$ is not contained in $\mathcal T_n$ for all $j$. }
\end{cases}
\end{aligned}
$$
We have $s_0 = 0 = t_0$ and $s_1 = 0$,  while $t_1 = 2$ and $s_2 = 1$.  
The sequences $\{s_n\}$ and $\{t_n\}$ are increasing in the sense that $s_n \le s_{n+1}$
and $t_n \le t_{n+1}$ for each $n \ge 0$.

Theorem~\ref{valideals} implies that for each positive integer $n$ there
exists a positive integer $\tau_n$ such that 
$$ 
\mathcal T_{\tau_n} ~ \subseteq ~ \mathcal S_n \quad \text{that is, } \quad 
W ~\in ~\mathcal T_{\tau_n} ~~ \implies ~~ W ~\in \mathcal S_n.
$$
Thus    $\tau_n$ is  any integer such that $\tau_n \ge s_n$.     The proof of Theorem~\ref{valideals}
also proves that:
\begin{equation} \label{discuss2}
s_n~\le ~ \min \{ ~j \in \mathbb N_0 ~| ~ I_{\mu}R_j \text{ is principal  for all  $\mu$ with  }  0 \le   \mu \le n-1  \}.
\end{equation}

\end{discussion}

 Theorem~\ref{approx}   applies to a directed family   $\{( R_n, \m_n) \}_{n \ge 0}$  of Noetherian local domains 
 such that $\bigcup_{n=0}^{\infty} R_n$ is a rank one valuation domain.

\begin{theorem} \label{approx}
Let $\{( R_n, \m_n) \}_{n \ge 0}$ be an infinite family of Noetherian local domains such that $R_n$ is a subring 
of $R_{n+1}$ for each integer $n \ge 0$.
Assume that $\ord_{R_n}$ defines a valuation for each integer $n \ge 0$, and  
that $\bigcup_{n=0}^{\infty} R_n = V$ is a rank one valuation domain.
Let $v$ denote a valuation associated to $V$,  and   let 
$f, g$ be nonzero elements in the maximal ideal of $V$.
Then for $n$ sufficiently large, $f$ and $g$ are in $\m_n$, and we have
$$
\frac{v (f)}{v (g)} ~= ~ \lim_{n \rightarrow \infty} \frac{\ord_{R_n} (f)}{\ord_{R_n} (g)}.
$$
\end{theorem}

\begin{proof}
Let  $L = \frac{v (f)}{v (g)}$ and let  $\epsilon > 0$. There exist
positive integers $p$ and $q$ such that 
$$ 
L - \epsilon ~< ~  \frac{p}{q} ~  \le  ~L ~ < ~ \frac{p+1}{q} ~ < ~ L + \epsilon.
$$
Then 
$$
\frac{p}{q} ~\le ~ L = \frac{v(f)}{v(g)} ~ < ~ \frac{p+1}{q}
$$
implies that
$$
pv(g) ~ \le ~ qv(f) ~< ~(p+1)v(g) \qquad \text{ and } \qquad v(g^p)~ \le ~ v(f^q) ~ < ~v(g^{p+1}).
$$  
For $n$ sufficiently large, we have   $\frac{f^q}{g^p}, \frac{g^{p+1}}{f^q} \in R_n$.  Therefore
	$$\ord_{R_n} (g^p) \le \ord_{R_n} (f^q) < \ord_{R_n} (g^{p+1}).$$
It follows that 
	$$\frac{p}{q} \le \frac{\ord_{R_n} (f)}{\ord_{R_n} (g)} < \frac{p+1}{q}.$$
We conclude that 
	$$L - \epsilon < \frac{\ord_{R_n} (f)}{\ord_{R_n} (g)} < L + \epsilon$$
for all sufficiently large integers $n$.
\end{proof}

Corollary~\ref{approxcor} is immediate from Theorem~\ref{approx}.  

\begin{corollary}   \label{approxcor}
Let $\{(R_n, \m_n)\}_{n \ge 0}$ be an infinite directed family of local quadratic transforms of a regular local ring  $R$.
Assume that $\{ (R_n, \m_n) \}_{n \ge 0}$ switches strongly infinitely often, 
so  $V = \bigcup_{n = 0}^{\infty} R_n$ is a rank $1$ valuation.
Let $v$ be a valuation associated to $V$.
Then for nonzero $f, g \in \m$,  we have
$$
\frac{v (f)}{v (g)} ~= ~ \lim_{n \rightarrow \infty} \frac{\ord_{R_n} (f)}{\ord_{R_n} (g)}.
$$
\end{corollary}

\section{Directed unions of monomial local quadratic transforms}\label{c3}

\begin{definition} \label{3.1}
Let $(R, \m)$ be a  $d$-dimensional regular local ring  with $d \ge 2$,   and fix 
$d$ elements $x,y,\ldots, z$ such that 
$\m:=(x, y, \ldots, z)R$.  
An element $p \in R$ is called a {\bf {monomial} in $x, y, \ldots, z$}
if there exists $(\alpha, \beta, \ldots, \gamma)\in \mathbb N^d_0$ such that $p=x^{\alpha}y^{\beta}\cdots z^{\gamma}$.

 An ideal $I$ of $R$  
is said to be a {\bf monomial ideal}  if $I$ is generated by monomials.  
     Let 
$$
^xA ~= ~ R \big[ \frac{\m}{x} \big] = R \big[ \frac{y}{x}, ~\ldots, ~\frac{z}{x} \big] \qquad x_1 ~:= ~x,~\quad y_1:~= ~\frac{y}{x},~\quad \ldots,~\quad~z_1~:=\ ~\frac{z}{x}.
$$
If $I$ is a monomial ideal in $R$,   the transform of $I$ in $^xA$ is  generated by elements of the
form $x_1^ay_1^b \cdots z_1^c $ with $a,b,\ldots, c \in \N_0$.  
This motivates us to define an ideal $J$ of $^xA$ to be
a {\bf monomial ideal}  if $J$ is generated by monomials in 
$x_1, y_1, \ldots, z_1$.   We consider  
monomial quadratic transformations of $R$ defined as follows:   the ring 
$^xR=R\big[\frac{\m}{x}\big]_{(x,~ \frac{y}{x},~\ldots,~\frac{z}{x})}$ is a 
{\bf monomial local quadratic transformation}  of $R$ {\bf in the  $x$-direction}.
An ideal $J$ of $^xR$ is said to be a {\bf monomial ideal} if $J$ is generated by
monomials in $x_1, y_1, \ldots, z_1$. 

In a similar manner, we define   $^yR$, $\ldots$, $^zR$ to be the  
{\bf local  monomial quadratic transformations}  
of $R$ in the {\bf  $y$-direction}, $\ldots$, {\bf  $z$-direction}, respectively, where   
$$ 
^yR~=  ~R\big[\frac{\m}{y}\big]_{(\frac{x}{y},~y,~\ldots,~\frac{z}{y})}, \quad~
\ldots, \quad~
^zR=R\big[\frac{\m}{z}\big]_{(\frac{x}{z},~ \frac{y}{z},~\ldots,~z)}.
$$
We define an ideal of $^yR$, $\ldots$, $^zR$ to be a {\bf monomial ideal} if it is 
generated by monomials in the respective rings.  We refer to the elements in the  fixed set of 
minimal generators of the regular local ring as
{\bf variables}.

For a monomial ideal $I$ of one of these rings, there exists a unique set of 
monomial minimal generators of $I$, \cite[Corollary~4]{KS}. We 
 let  $\Delta (I)$ 
denote the set of monomial minimal generators of $I$.

\end{definition}

\begin{setting}\label{3.2a} 
Let $(R,\m)$ and $x, y, \ldots, z$ be as in Definition~\ref{3.1}. 
   For $n$ a positive integer,   consider a  sequence of regular local rings 
\begin{equation}
\{(R_i, \m_i)\}_{i =0}^n~\equiv~R  ~=: ~  R_0 ~\subset~ R_1 ~\subset~ R_2~ \subset \cdots \subset~ R_n ,
\end{equation}
where $R_{i+1}$ is a  monomial  local quadratic transformation of $R_i$,  for each $i < n$ 
$$ 
\m_i=(x_i, y_i, \ldots, z_i)R_i \quad \text{ and } \quad 
\m_{i+1}=(x_{i+1}, y_{i+1}, \ldots, z_{i+1})R_{i+1},
$$
and where one of the following $d$-cases occur:
$$
x_{i+1}:=x_i,\quad y_{i+1}:=\frac{y_i}{x_i},\quad \cdots \quad ,z_{i+1}:=\frac{z_i}{x_i},
$$
in which case we say the transform from $R_i$ to $R_{i+1}$ is {\bf{in the $x$-direction}}, or
$$
x_{i+1}:=\frac{x_i}{y_i},\quad y_{i+1}:=y_i,\quad \cdots \quad ,z_{i+1}:=\frac{z_i}{y_i}, 
$$
in which case we say the transform from $R_i$ to $R_{i+1}$ is {\bf{in the $y$-direction}}, or
$$\vdots$$
$$
x_{i+1}:=\frac{x_i}{z_i},\quad y_{i+1}:=\frac{y_i}{z_i},\quad ~\cdots \quad ,z_{i+1}:=z_i, 
$$
in which case we say the transform from $R_i$ to $R_{i+1}$ is {\bf{in the $z$-direction}}.

Let $\mathfrak D_{x}$(respectively,$~ \mathfrak D_{y},~\ldots, ~ \mathfrak D_{z}$) 
 denote the set of  integers $i   \in \{0, 1, \ldots, n-1\}$  for which the transform from $R_i$ to $R_{i+1}$ 
is in the $x$-direction (respectively, the $y$-direction, ~\ldots, ~ the $z$-direction).
\end{setting}

\begin{theorem}\label{3.3a}
Let notation be as in Setting~\ref{3.2a}.  The following are equivalent:  
\begin{enumerate} 
\item  The sets  $\mathfrak D_{x}, ~\mathfrak D_{y},~\ldots,~ \mathfrak D_{z}$ are all nonempty.
\item  For every nonzero $f \in \m$, we have  $\ord_{R_n}(fR)^{R_n}  <  \ord_R(f)$.
\end{enumerate}
\end{theorem}

\begin{proof}
Assume that item ~1 fails. Then at least one of the sets 
$\mathfrak D_{x}, ~\mathfrak D_{y},~\ldots,~ \mathfrak D_{z}$ is empty. By relabeling, we may assume
that  $\mathfrak D_x$ is empty. Then $R_n \subset R_{xR}$ and  we have
$\ord_{R_n}(xR)^{R_n}   = 1 =   \ord_R(x)$.

Assume that item~1 holds and fix a nonzero element   $f \in \m$.  
Let $r:=\ord_{R}(f) > 0$. 
The order of the transform of $fR$ never increases, cf \cite[Lemma~3.6]{HKT}. 
By \cite[Proposition~1.3]{HubS}, there exist monomials $g_1, \ldots, g_t$  
and units  $\lambda_1, \ldots, \lambda_t$ in $R$  such that 
$$
f~=~ \lambda_1 g_1+ \cdots +~\lambda_t g_t.
$$
Let $g_i=x^{a_{i}}y^{b_{i}} \cdots z^{c_{i}}$ be a 
monomial with $\ord_{R}(g_i) = r$. By relabeling the variables, we may assume $a_{i} > 0$. 
Since the set $\mathfrak D_{x}$ is nonempty, there 
exists $s  \in \{0, \ldots, n - 1\}$ such that the transform from $R_{s}$ to $R_{s+1}$ 
is in the $x$-direction. Let $s$ be minimal with this property. 

If $\ord_{R_s} (f R)^{R_s} < r$, we are done.
Otherwise, we have  $\ord_{R_s} (f R)^{R_s} = r$, and there exists a 
monomial $q$ in $x_s, y_s, \ldots, z_s$ such that
$$
(f R)^{R_s} = (\frac{f}{q}) R_s = (\frac{g_1}{q} \lambda_1 + \ldots + \frac{g_t}{q} \lambda_t) R_s.
$$
The elements $\frac{g_1}{q}, \ldots, \frac{g_t}{q}$ are  distinct monomials in $x_s, y_s, \ldots, z_s$.
Since $\ord_{R_s}$ is a monomial valuation and $\ord_{R_s} (f R)^{R_s} = r$, 
we have $\ord_{R_s} (g_i R)^{R_s} = r$. Thus 
$$
(g_i R)^{R_s} ~ = ~ (\frac{g_i}{q}) R_s ~ = ~ (x_s^{a_{i}} y_s^{b_i} \cdots z_s^{c_i}) R_s.
$$
Since the transform from $R_{s}$ to $R_{s+1}$ is in the $x$-direction, 
$\ord_{R_{s+1}} (g_i R)^{R_{s+1}} = r - a_i < r$.
Set $q' = q x_s^{r}$. Then 
$$
(f R)^{R_{s+1}} ~=~  (\frac{f}{q'}) R_{s+1} ~ = ~ (\frac{g_1}{q'} \lambda_1 + \ldots + \frac{g_t}{q'} \lambda_t) R_{s+1}.
$$
We have  $\frac{g_i}{q'} R_{s+1} = (g_i R)^{R_{s+1}}$.  Since $\ord_{R_{s+1}}$ 
is a monomial valuation and the monomials $\frac{g_1}{q'}, \ldots, \frac{g_t}{q'}$ are distinct, 
we have 
$\ord_{R_{s+1}} (f R)^{s+1} \le r - a_i < r$.
This completes the proof of Theorem~\ref{3.3a}.
\end{proof}

\begin{remark}  \label{3.15}
For the transform of an ideal as defined in Definition~\ref{1.1}, we
observe some properties of the transform of a monomial ideal 
with respect to a monomial local quadratic transform from $R_0$ to $R_1$.

Let $(R_1, \m_1)$ be the monomial local quadratic transformation of $R_0$ in the $w$-direction,
where $w \in \Delta(\m_0)$, and $\m_1:=(x_1, y_1, \ldots, w_1, \ldots, z_1)$, where
$$
x_1~:= ~\frac{x}{w},\quad y_1~:=~\frac{y}{w},~ \ldots~,~ w_1~:=~ w, ~\ldots ~ , ~z_1~:= ~\frac{z}{w}.
$$
Let $p:=x^{\alpha_1}y^{\beta_1}\cdots w^{\eta_1}\cdots z^{\gamma_1}$ and 
$q:=x^{\alpha_2}y^{\beta_2}\cdots w^{\eta_2}\cdots z^{\gamma_2}$ be
distinct  monomials 
in $R$ and let $I := (p,q)R$.  Assume that $r:= \ord_R(p) \le \ord_R(q) := s$. 
Thus   $r=\ord_{R}(I)$. 
In $R_1$, the monomials $p$ and $q$ can be written as follows:
$$
\aligned 
p:&=x^{\alpha_1}y^{\beta_1}\cdots w^{\eta_1}\cdots z^{\gamma_1}\\
  &=(w x_1)^{\alpha_1}(w y_1)^{\beta_1} \cdots w^{\eta_1} \cdots (w z_1)^{\gamma_1}\\
  &=w_1^{r}(x_1^{\alpha_1}y_1^{\beta_1}~\cdots w_1^{0}\cdots~ z_1^{\gamma_1}) ~=~w_1^rp_1 ~=~w^rp_1
\endaligned
$$
and 
$$
\aligned 
q:&=x^{\alpha_2}y^{\beta_2}\cdots w^{\eta_2}\cdots z^{\gamma_2}\\
  &=(w x_1)^{\alpha_2}(w y_1)^{\beta_2} \cdots w^{\eta_2} \cdots (w z_1)^{\gamma_2}\\
  &=w_1^{r}(x_1^{\alpha_2}y_1^{\beta_2}~\cdots w_1^{s-r}\cdots~ z_1^{\gamma_2}) ~=~w_1^rq_1 ~=~w^rq_1.
\endaligned
$$
Thus $I_1 := I^{R_1} = (p_1, q_1)R_1$ is the transform of $I$ in $R_1$. We have  
\begin{equation} \label{3.151}
\ord_{R_1}(I_1) \le \ord_{R_1}(p_1)
 \qquad 
\text{ and } \quad \ord_{R_1}(p_1) = r - \eta_1 \le \ord_R(p).
\end{equation}
Hence 
if $\eta_1 > 0$, then $\ord_{R_1}(I_1) < \ord_R(I)$.

It is observed in \cite [Lemma~3.6]{HKT} 
that $\ord_{R_1}(I^{R_1})~ \leq~ \ord_{R_0}(I)$ even without the assumption that $I$ is
a monomial ideal.
\end{remark}

\begin{remark}\label{3.41}
Let notation be as in Setting~\ref{3.2a}, and let $V$ be a valuation domain that dominates $R_n$. 
The following are equivalent:
\begin{enumerate} 
\item The transform from $R_0$ to $R_1$ is in the $x$-direction.
\item $v(x) < v(w)$ for each $w \in \Delta(\m)$ with $w \ne x$. 
 \item $\m V = xV$.
\end{enumerate}
Assume that 
the transform from $R_0$ to $R_1$ is in the $x$-direction.  By 
relabeling the variables, we may assume that 
$$ 
v(y)~ = ~\min\big \{v(w) ~|~ w \in \Delta(\m) \quad \text{ with }~~ w \ne x \big\}.
$$
With this assumption, the   following are equivalent:
\begin{enumerate} 
\item[(a)] The set $\frak D_y$ is nonempty.
\item[(b)] The set $\frak D_y \cup \cdots \cup \frak D_z$ is nonempty.
\item[(c)] There exists a positive integer $s$ in $\{ 1, \ldots, n - 1 \}$ such that the transform from $R_0$ to $R_s$ 
is in the $x$-direction and the transform from $R_s$ to $R_{s+1}$ is in the $y$-direction.
\item[(d)]  We have $sv(x) < v(y) < (s+1)v(x)$ for some  $s$ in $\{ 1, \ldots, n - 1 \}$.
\end{enumerate}
\end{remark}

\begin{remark}\label{3.43}
Let notation be as in Remark~\ref{3.41}.
Assume that the sets $\frak D_{x}, ~\frak D_{y},~\ldots,~ \frak D_{z}$ are all nonempty. 
Label the variables $x, y, \ldots, z$ in $\Delta(\m)$ as $w_1, w_2, \ldots w_d$,
where $i < j$ if the sequence $\{R_i\}_{i=0}^n$ contains a transform in the $w_i$-direction
before it contains a transform in the $w_j$-direction. Thus with the 
assumptions of Remark~\ref{3.41}, we have $w_1 = x$,  $w_2 = y$,  and 
$v(w_i) < v(w_j)$ if and only if $i < j$. 
\end{remark}

\begin{definition} Let $V$ be a valuation domain, and let $b_1, \ldots, b_s$ 
be nonzero elements of $V$. We say that $b_1, \ldots, b_s$ are {\bf comparable in} $V$
if $\sqrt{b_1V} = \cdots = \sqrt{b_sV}$, that is, the ideals $b_iV$ and $b_jV$
have the same radical for all $i$ and $j$ between $1$ and $s$.
\end{definition}

\begin{proposition}\label{3.44}
Let notation be as in Remark~\ref{3.43}.
Thus the sets $\frak D_{x}, ~\frak D_{y},~\ldots,~ \frak D_{z}$ are all nonempty,
and our  fixed  regular system of parameters $(x, y, \ldots, z)$ is also denoted 
$(w_1, w_2, \ldots, w_d)$.  
Let $a_i := v(w_i)$ for $1 \le i \le d$. Thus 
the tuple $(a_1, a_2, \ldots, a_d)$ is the family of 
$v$-values of $(w_1, w_2, \ldots, w_d)$. 
Then we have:
\enumerate
\item $0~<~a_1~<~a_2~<~ \cdots~<~a_d$.
\item There exists a positive integer $s$ such that $sa_1~< ~a_2~<~(s+1)a_1$.
\item We have 
 $$(j - 2) a_j~<~a_1+a_2+\cdots+a_{j-1}\quad \text{ for $j$ with $~~3~ \leq~ j~ \leq ~ d$}.
$$
\item The elements $w_1, \ldots, w_d$ are comparable in $V$.
\endenumerate
\end{proposition}

\begin{proof}
Item 1 follows from Remark~\ref{3.43}, and item~2 follows from Remark~\ref{3.41}. 

In order to give a detailed proof of item~3, we introduce more notation. 
For each  $i \in \{0, 1, \ldots, n\}$ and each positive integer $j$ with $1 \le j \le d$, 
we define $w_{i_j}$ as follows:
$$
\aligned 
\m = &\m_0 ~ = ~ (x, y, \ldots, z)R_0 ~=~  (w_{0_1}, w_{0_2}, \ldots, w_{0_d})R_0\\
& \m_i ~=~ (x_i, y_i, \ldots, z_i)R_i ~= ~ (w_{i_1}, w_{i_2}, \ldots, w_{i_d})R_i. \\
\endaligned
$$
To prove item~3, it suffices to show that if  $(\ell -2)a_{\ell}~\ge ~a_1+a_2+\cdots+a_{\ell-1}$ for some $\ell$ with $~~3~ \leq~ \ell~ \leq ~ d$,
then  the set $\frak D_{w_{\ell}}$ is empty.
Let 
$$
\aligned
& \big(a_{0_1}, a_{0_2}, \ldots, a_{0_d}\big)=\big(v(w_{0_1}), v(w_{0_2}\big), \ldots, v(w_{0_d})),\\
&\big(a_{k_1}, a_{k_2}, \ldots, a_{k_d}\big)=\big(v(w_{k_1}), v(w_{k_2}), \ldots, v(w_{k_d})\big)\quad \text{ for each $k \geq 1$}.
\endaligned
$$
With this notation, we  show $(\ell-2) a_{k_\ell}~  ~\ge ~a_{k_1} + \cdots + a_{k_{\ell-1}}$ for all $k \ge 0$ by induction on $k$. By hypothesis   $(\ell - 2) a_{0_\ell}  ~\ge ~a_{0_1} + \cdots + a_{0_{\ell-1}}$.
Assume  the assertion holds for $k$ and let $e_k := \min\{a_{k_1}, a_{k_2}, \ldots, a_{k_{\ell-1}} \}$.
Then we have
$$
\aligned 
(\ell - 2) a_{k+1_\ell} ~&~ =(\ell-2) (a_{k_\ell} - e_k)~=~ (\ell - 2) a_{k_\ell} - (\ell - 2) e_k  \\
              & ~\ge ~a_{k_1} + \cdots + a_{k_{\ell-1}} - (\ell - 2)e_k \\
              & ~= ~ a_{k+1_1} + \cdots + a_{k+1_{\ell-1}}. 
\endaligned
$$
Therefore the assumption that $\frak D_{w_{\ell}}$ is nonempty
implies that 
$$(\ell - 2) a_{\ell}~< ~a_1+a_2+\cdots+a_{\ell-1}.$$
This proves item~3.  Item~4 now follows from items~1, 2 and 3. \end{proof}

\section{Infinite Directed Unions of Monomial Local Quadratic Transforms}  \label{c4}

 \begin{remark}\label{0.2}
 Let $(R, \m)$ be a $d$-dimensional regular local ring  with $d \ge 2$,  and let $V$ be a
 valuation domain birationally dominating $R$ and zero-dimensional 
over $R$.   Let 
$v$ be a valuation associated to $V$. Then there exists a
uniquely defined  infinite sequence  
\begin{equation} \label{0.3}
\{(R_n, \m_n)\}_{n\geq 0}~\equiv~R=:R_0 ~\subset~ R_1 ~\subset~ R_2~ \subset \cdots \subset~ R_n~ \subset~ \cdots \subset  V
\end{equation}
of local quadratic transforms of $R$ along $V$.

\begin{enumerate}
\item If the $v$-values $v(x), v(y), \ldots, v(z)$ of a
regular system of parameters for $\m$ 
are rationally independent real numbers, then the sequence given 
in Equation~\ref{0.3} is monomial with respect to this regular system of parameters.   However,  as Shannon shows in \cite[Example~4.17]{S},  
in this situation it is often
the case that $S := \bigcup_{n \ge 0}R_n$ is properly contained in $V$. 
\item\label{0.2.2}
If $v$ is   rank one  discrete,  then the 
  sequence given in Equation~\ref{0.3}  switches strongly infinitely often,  that is  $V=\bigcup_{n\geq 0}R_n$.
To see this, let $\frac{p}{q} \in V$, with $p, q \in R$.
Assume that from $R$ to $R_1$, we divide by $x$, where $v (x) > 0$.
Let $p_1 = p / x$ and $q_1 = q / x$, so that $\frac{p}{q} = \frac{p_1}{q_1}$, where $v (q_1) < v (q)$.
A simple induction argument shows that we can write $\frac{p}{q} = \frac{p_n}{q_n}$, 
where $p_n, q_n \in R_n$ and $v (q_n) = 0$, for some $n$.
See for example \cite[Proposition~23.3]{HRW}.
Thus the sequence $\{(R_n, \m_n)\}_{n \ge 0}$ switches strongly infinitely often.

\item
If  $V$ is rank one and  the rational rank of $V$ is less than $d$, it is possible that the
field extension $R/\m \subset V/\m_v$ is an infinite algebraic field 
extension.   This   is illustrated  in Example 4 of page 104 in \cite{ZS2}.

\item\label{0.2.4}
Assume that  $V$ has  rank one and rational rank $r$ and that  the  local quadratic  sequence  
 $\{(R_n, \m_n)\}_{n \ge 0}$  along $V$ switches strongly infinitely often.  
Then for each sufficiently large integer $n$, there exists a sequence of elements $w_1, \ldots, w_r $  
that are part of a regular system of parameters for $\m_n$  and are 
such that  $v (w_1), \ldots, v (w_r)$ are rationally independent.

\begin{proof}
Since $V$ has rational rank $r$, there exist  elements $a_1, \ldots, a_r $   in  $R$ 
such that $v (a_1), \ldots, v (a_r)$ are rationally independent.  
Since $V = \bigcup_{n \ge 0}R_n$,  it follows from \cite[Prop. 4.18]{S} that for sufficiently large integers $n$, the product
$a_1 \cdots a_r$ is a monomial with respect to a regular system of parameters $w_1, \ldots, w_d$ of $R_n$.
It follows that $v(a_1), \ldots, v(a_r)$ are in 
the $\Q$-linear subspace  $L$  of $\mathbb{R}$ spanned by $v (w_1), \ldots, v (w_d)$.  Since   $v (a_1), \ldots, v (a_r)$ are 
linearly independent over $\Q$,    
the vector space  $L$ has rank at least $r$.  By re-labeling the   $w_i$,   
 it follows that   $v(w_1), \ldots,  v(w_r)$   are rationally independent.
\end{proof}
\end{enumerate}
\end{remark}

We observe in Proposition~\ref{maxratrank} that an infinite sequence $\{(R_n, \m_n)\}_{n \ge 0}$ of local
quadratic transforms that switches strongly infinitely often is eventually monomial
if  
$V = \bigcup_{n \ge 0}R_n$ has maximal rational rank .

\begin{proposition} \label{maxratrank}
Let $(R, \m)$, be a $d$-dimensional regular local ring and let $V$ be a rank $1$, rational rank $d$ valuation domain birationally dominating $R$.
Consider the infinite sequence
	$$R =: R_0 \subset R_1 \subset \cdots \subset R_n \subset \cdots \subset V$$
of local quadratic transforms of $R$ along $V$.   If $V = \bigcup_{n=0}^{\infty} R_n$,  then 
there exists an integer $n  \ge  0$ and a regular system of parameters $x, y, \ldots, z$ of $R_n$ such that the sequence  $\{R_i\}_{i \ge n}$ is monomial 
in these regular parameters.
\end{proposition} 

\begin{proof}
By Remark~\ref{0.2}.\ref{0.2.4}, there exists an integer $n \ge 0$ and a regular system of parameters $x, y, \ldots, z$ of $R_n$ such that $v (x), v (y), \ldots, v (z)$ are rationally independent.
It follows that the sequence $\{R_i\}_{i \ge n}$ is monomial  with respect to $x,y, \ldots, z$.
\end{proof}

In Setting~\ref{3.2}, we extend the notation of Setting~\ref{3.2a} to an infinite sequence.

\begin{setting}\label{3.2} Let $(R,\m)$ and $x, y, \ldots, z$ be as in Setting~\ref{3.2a}. 
  We consider an infinite  sequence of regular local rings 
\begin{equation}
\{(R_i, \m_i)\}_{i\geq 0}~\equiv~R=:R_0 ~\subset~ R_1 ~\subset~ R_2~ \subset \cdots \subset~ R_i~ \subset~ \cdots ,
\end{equation}
where $R_{i+1}$ is a monomial local quadratic transformation of $R_i$ for each $i \ge 0$. 
As in Setting~\ref{3.2a}, let $\frak D_{x}$(respectively,$~ \frak D_{y},~\ldots, ~ \frak D_{z}$) 
 denote the set of nonnegative integers $i$ for which the transform from $R_i$ to $R_{i+1}$ 
is in the $x$-direction (respectively, the $y$-direction, ~\ldots, ~ the $z$-direction).

Set $S:=\bigcup_{i\geq 0} R_i$. It is straightforward to prove that 
 $S$ is a normal local domain with field of fractions $\mathcal Q(R)$ and 
maximal ideal $\m_S:=\bigcup_{i\geq 0}\m_i$. Our assumption that the local 
quadratic transforms are monomial implies that there is
no residue field extension, that is, $\bigcup_{i\geq 0}(R_i/\m_i) = R/\m$.
Moreover, if $S$ is not a valuation domain, it is known that 
there exist infinitely many valuation domains $V$ that birationally
dominate $S$ and have positive 
dimension over $S$, that is, the residue field of $V$ has positive transcendence
degree over the field $S/\m_S$; see   \cite[Lemma~7]{A} or  \cite[Exercise~6.24]{SH}. 
\end{setting}

\begin{theorem}\label{3.3}
Let notation be as in Setting~\ref{3.2}, and let $V$ be a valuation domain that 
birationally dominates $S$. The following are equivalent:  
\begin{enumerate} 
\item  The sets  $\frak D_{x}, ~\frak D_{y},~\ldots,~ \frak D_{z}$ are all infinite.
\item  $\{(R_n, \m_n)\}_{n \ge 0}$ switches strongly infinitely often.  
\item $S = V$ and $V$ has rank 1.  
\end{enumerate}
\end{theorem}

\begin{proof}
$(2) \Rightarrow (1)$ : Assume that item ~1 fails. Then at least one of the sets 
$\frak D_{x}, ~\frak D_{y},~\ldots,~ \frak D_{z}$ is finite. By relabeling, we may assume
that  $\frak D_x$ is finite. Then there exists a positive integer $j$
such that for $n \ge j$, the transform from $R_n $ to $R_{n+1}$ is not in the $x$-direction. 
Let $\m_j:=(x_j, y_j, \ldots, z_j)R_j$. Then  $\p_j:=(x_j)R_j$ has the property that 
$\bigcup_{n\geq 0}R_n \subseteq \big(R_j\big)_{\p_j}$. 
Thus $\{(R_n, \m_n)\}_{n \ge 0}$ does not switch strongly infinitely often.

$(1) \Rightarrow (2)$ :
Assume that item~1 holds.  To prove that 
$\{(R_n, \m_n)\}_{n \ge 0}$ switches strongly infinitely often 
it suffices to show for each integer $j \ge 0$ and nonzero element $f \in R_j$ 
that there exists an integer $n \ge j$ such that the  
transform $(fR_j)^{R_n} = R_n$, cf.  Remark~\ref{switch}.
Let $r:=\ord_{R_j}(f)$. To prove that the transform $(fR_j)^{R_n} = R_n$ for some $n$,
it suffices to prove  that the order of the transform 
$(fR_j)^{R_s}$  is less than $r$ for some integer $s > j$, and this is
immediate from Theorem~\ref{3.3a}. 
$(2) \Rightarrow (3)$ : This is clear by Remark \ref{rank}.
$(3) \Rightarrow (2)$ : This is shown by Granja \cite[Theorem~13]{Gr}.  
\end{proof}

\begin{lemma}\label{3.6}
Let notation be as in Setting~\ref{3.2}, and let  
$p$ and $q$ be distinct  monomials   in $R$.  
If the  sets $\frak D_{x}, ~\frak D_{y},~\ldots,~ \frak D_{z}$ are all infinite, then
there exists an  integer $t$ such that either  $p/q \in \m_t$ or $q/p \in \m_t$.
\end{lemma}

\begin{proof}
Let $I:=(p, q)R$.  We prove Lemma~\ref{3.6} by induction on  $r := \ord_R(I)$. 
The case where $r = 0$ is clear. Assume the assertion of  Lemma~\ref{3.6} holds
for all nonnegative integers less than $r$.

We may assume $\ord_R(p) \le \ord_R(q)$. 
Then $r = \ord_R(p)$. Assume notation as in Remark~\ref{3.15}. Thus the
transform from $R$ to $R_1$ is in the $w$-direction. If $\eta_1 > 0$, then  as 
shown in Equation~\ref{3.151}, we have 
$\ord_{R_1}(p_1) < r$.  Hence  by our inductive hypothesis applied to the ideal $I_1$ in
 $R_1$, we have,  for some positive integer $t$,      
  either $\frac{p}{q} = \frac{p_1}{q_1} \in \m_t$
or $\frac{q}{p} = \frac{q_1}{p_1} \in \m_t$.  

If $\eta_1 = 0$, then $\ord_{R_1}(p_1) = \ord_R(p) = r$. If $\ord_{R_1}(q_1) < r$,
then again by induction applied to the ideal $I_1$ in $R_1$, we have 
either $\frac{p}{q} = \frac{p_1}{q_1} \in \m_t$
or $\frac{q}{p} = \frac{q_1}{p_1} \in \m_t$ for some positive integer $t$.

If $\eta_1 = 0$ and $\ord_{R_1}(q_1) \ge r$,  then the ideal $I_1 = (p_1,q_1)R_1$  
satisfies the same hypothesis as the ideal $I = (p,q)R$.   Let $I_2 = (p_2, q_2)R_2$ denote
the transform of $I_1$ in $R_2$,  where $\frac{p_2}{q_2} = \frac{p}{q}$.  
  If $\ord_{R_2}(I_2) <  r$,  then 
for some positive integer $t$,      
  either $\frac{p}{q} = \frac{p_2}{q_2} \in \m_t$
or $\frac{q}{p} = \frac{q_2}{p_2} \in \m_t$.   If  $\ord_{R_2}(I_2) = r$,
we define $I_3 = (p_3, q_3)R_3$ and continue. Let $I_i = (p_i, q_i)R_i$ denote
the transform of $I$ in $R_i$, where $\frac{p_i}{q_i} = \frac{p}{q}$, for each $i \ge 1$. 

If $\ord_{R_i}(I_i) = r$ for all $i$, then the transform from $R_{i-1}$ to $R_i$ is
never in a direction where the associated exponent of $p$ is positive.
This is impossible since the
sets $\frak D_{x}, ~\frak D_{y},~\ldots,~ \frak D_{z}$ are all infinite
and the monomial $p$ contains at least one positive exponent. 
We conclude by induction that  there exists an  integer $t$ such that either  $p/q \in \m_t$ or $q/p \in \m_t$.
\end{proof}

\begin{corollary} \label{ratrank} Assume notation as in Theorem~\ref{3.3}. The
valuation domain $V$ has rational rank $d = \dim R$, and $V$ is a zero-dimension over $R$.
If $v$ is a valuation associated to $V$, then 
$v(x), v(y), \ldots, v(z)$ are rationally independent.  In particular,  $v$ is a
monomial valuation
\end{corollary} 

\begin{proof} 
Assume that 
$$
a_1v(x)+b_1v(y)+\cdots +c_1v(z)=a_2v(x)+b_2v(y)+\cdots +c_2v(z),
$$
with $(a_1, b_1, \ldots, c_1)$ and $(a_2, b_2, \ldots, c_2)$ being $d$-tuples 
of nonnegative integers. 
Then we have 
$$
v(x^{a_1}y^{b_1}\cdots z^{c_1})=v(x^{a_2}y^{b_2}\cdots z^{c_2}),
$$ 
and hence $x^{a_1}y^{b_1}\cdots z^{c_1}=x^{a_2}y^{b_2}\cdots z^{c_2}$ by Lemma~\ref{3.6}.
Thus we have
$$
a_1=a_2,~ b_1=b_2,~ \ldots,~ c_1=c_2.
$$
By \cite[Theorem~6.6.7]{SH}, we have
$$
 \rat.rank v + \tr.deg_{R/\m} k(v) \leq \dim(R).
$$
Since  $ d \leq \rat.rank v$,  we have $\rat.rank v=d$. 
It follows that $\tr.deg_{R/\m} k(v)=0$.
\end{proof}

\section{Infinite Directed Unions of Local Quadratic Transforms}  \label{c5}

Related to Item~\ref{0.2.2} of Remark \ref{0.2}, Example~\ref{NotUnionRR1}  demonstrates  the existence  
of  a nondiscrete rational rank $1$ valuation domain $V$ that birationally dominates and is zero-dimensional 
over a $3$-dimensional   regular local ring $R$  and  is such  that  
the sequence $\{(R_n, \m_n)\}_{n \ge 0}$   of local quadratic transforms of $R$ along $V$  
as  in Equation~\ref{0.3} fails to switch strongly infinitely infinitely often, 
that is, we may have $\bigcup_{n\ge0} R_n \subsetneq V$.

\begin{example}\label{NotUnionRR1}
Let $x, y$ be indeterminates over a field $k$,  and let $A_0 = k [[x, y]]$ 
be the formal power series ring in $x$ and $y$ over $k$. 
We describe a sequence  $\{ (A_n, \n_n) \}_{n \ge 0}$   of local quadratic transforms  of $A_0$, 
such that  $\n_n = (x_n, y_n)A_n$ is the maximal ideal of $A_n$, 
where for each even integer $n \ge 0$, we set $A_{n+1} = A_n[\frac{y_n}{x_n}]_{(x_n, \frac{y_n}{x_n}-1)}$,   
and for each odd
integer $n \ge 1$, we set   $A_{n+1} = A_n[\frac{x_n}{y_n}]_{(y_n, \frac{x_n}{y_n})}$.
Thus for each integer $k \ge 0$,  we have
$$
x_{2k+1} = x_{2k}, \quad    y_{2k+1} = \frac{y_{2k} - x_{2k}}{x_{2k}},  \quad  x_{2k+2} = \frac{x_{2k+1}}{y_{2k+1}}, \quad \text{ and  } \quad 
y_{2k+2} = y_{2k+1}.
$$
Notice that   $A_{n+1}$ is the localization of $A_n [x_{n+1}, y_{n+1}]$ at the maximal ideal generated by $x_{n+1}, y_{n+1}$.

Let $W = \bigcup_{n=0}^{\infty} A_n$.  By Remark~\ref{3.25},  $W$ is a valuation domain that birationally dominates $A_0$ and has
residue field $k$.    Let $w$ denote the valuation associated to $W$ such that $w(x)= 1$.  The fact that $W$ dominates  every $A_n$ implies
that  $W$ is  rational rank $1$ nondiscrete  and that  $w (x_{2k}) = w (x_{2k+1}) = w (y_{2k}) = w (y_{2k-1}) = \left( \frac{1}{2} \right)^k$,
for all $k \ge 0$.
Since for $k \ge 0$, $w (\n_{2k}) = \left( \frac{1}{2} \right)^{k}$ and $w (\n_{2k+1}) = \left( \frac{1}{2} \right)^{k+1}$,
it follows that
	$$\sum_{k \ge 0} w (\n_k) = \sum_{k \ge 0} \left( \frac{1}{2} \right)^k + \sum_{k \ge 0}  \left( \frac{1}{2} \right)^{k+1} = \frac{1}{1 - \frac{1}{2}} + \frac{\frac{1}{2}}{1 - \frac{1}{2}} = 3.$$

Let $p (x) \in k [[x]]$ be such that $x$ and $p (x)$ are algebraically independent over $k$ and let $z = x^4 p (x)$.
Let $R = k [x, y, z]_{(x, y, z)}$, $V = W \cap k (x, y, z)$, and $v = w|_{R (x, y, z)}$.
It follows that $V$ is a rational rank 1 nondiscrete valuation domain that birationally dominates $R$ and has residue field $k$.
Let $\{(R_n, \m_n) \}_{n \ge 0}$ be the sequence of local quadratic transforms of $R$ along $V$.
Since $v(z) \ge 4$,  we have $\m_n = (x_n, y_n, z_n)$, where $x_n, y_n$ are the generators of $\n_n$ 
as defined above and $z_n$ is defined inductively by $z_{n+1} = \frac{z_n}{x_n}$ for $n$ even and $z_{n+1} = \frac{z_n}{y_n}$ for $n$ odd.
Since we never divide in the $z$-direction, $\bigcup_{n\ge0} R_n \subset R_{(z R)}$, and thus $\bigcup_{n \ge 0} R_n \subsetneq V$.
\end{example}

By stringing together infinitely many finite sequences of monomial local quadratic transforms 
as in Theorem~\ref{3.3a} and interspersing between each two a 
local quadratic transform that is not monomial, we obtain
examples of infinite sequences $\{R_i\}_{i \ge 0}$  that switch strongly infinitely often and have the property that 
$V = \bigcup_{i \ge 0}R_i$ has rational rank less than $d = \dim R$. 
We give  explicit examples with $d = 3$ in Example~\ref{4.18}. We use the 
following remark.

\begin{remark} \label{4.175}  Let $(R,\m)$ be a Noetherian local domain, and 
let $V$ be a valuation domain dominating $R$. Let $v$ be a valuation associated to $V$.
Assume
that $\dim_{R/\m}\m/\m^2 = d$, and that there exist elements $x_1, \ldots, x_d$ in $\m$ 
such that $v(x_1) < v(x_2) < \cdots < v(x_d) < v(\m^2)$.
Then for each $y \in \m \setminus \m^2$, we have $v(y) \in \{v(x_1), \ldots, v(x_d) \}$.

To see this, let $R = I_0 \supsetneq I_1 \supsetneq I_2 \supsetneq \ldots$ be the sequence of valuation ideals of $V$ in $R$, where $I_{n+1} = \{ a \in R \mid v (a) > v (I_n) \}$.
For $1 \le i \le d$, define $J_i = \{ a \in R \mid v (a) \ge v (x_i) \}$.
Define $J_0 = R$ and $J_{d+1} = \{ a \in R \mid v (a) \ge v (\m^2) \}$.
The assumption that $v (x_d) < v (\m^2)$ implies that $J_{d+1} \subsetneq J_d$.
Thus we have the chain of inclusions of ideals,
	$$R = J_0 \supsetneq J_1 \supsetneq \ldots \supsetneq J_d \supsetneq J_{d+1} \supset \m^2.$$
Since $\lambda_{R} (R / \m^2) = d + 1$, it follows that for $0 \le i \le d$, $J_i / J_{i+1}$ is a simple $R$-module, and it follows that $J_{d+1} = \m^2$.
Thus for $0 \le i \le d + 1$, $I_i = J_i$, and for each $y \in \m \setminus \m^2$, $v (y) \in \{ v (I_1), \ldots, v (I_d) \} = \{ v (x_1), \ldots, v (x_d) \}$.
\end{remark} 

\begin{example} \label{4.18}
Let $(R, \m)$ be a $3$-dimensional regular local ring, and let $\m = (x, y, z) R$.  Define: 

$$ 
\begin{aligned} 
R_1 &= R[\frac{\m}{x}]_{(x, \frac{y}{x}, \frac{z}{x})} \quad\text{ and } \quad \m_1 = (x_1, y_1, z_1)R_1 =  (x, \frac{y}{x}, \frac{z}{x})R_1\\
R_2 &= R_1[\frac{\m_1}{y_1}]_{(\frac{x_1}{y_1}, y_1, \frac{z_1}{y_1})} \quad\text{ and } \quad \m_2 = (x_2, y_2, z_2)R_2 =  (\frac{x^2}{y}, \frac{y}{x}, \frac{z}{y})R_2 \\
R_3 &= R_2[\frac{\m_2}{z_2}]_{(\frac{x_2}{z_2}, \frac{y_2}{z_2},  z_2)} \quad\text{ and } \quad \m_3 = (x_3, y_3, z_3)R_3 =  (\frac{x^2}{z}, \frac{y^2}{x z}, \frac{z}{y})R_3 \\
R_4 &= R_3[\frac{\m_3}{x_3}]_{(x_3, \frac{y_3}{x_3} - 1, \frac{z_3}{x_3} - 1)} \quad\text{ and } \quad \m_4 = (x_4, y_4, z_4)R_4 =  (\frac{x^2}{z}, \frac{y^2}{x^3} - 1, \frac{z^2}{x^2 y} - 1)R_4 
\end{aligned}
$$
For each valuation domain $V$ birationally  dominating $R_4$, we have  $v (y) = \frac{3}{2} v (x)$ and $v (z) = \frac{7}{4} v (x)$.

Starting from $(R_4, \m_4)$ and $\m_4 = (x_4, y_4, z_4)R_4$, 
we make a sequence $R_4 \subset R_5 \subset R_6 \subset R_7$ of monomial local quadratic transforms 
with respect to the fixed basis $x_4, y_4, z_4$ of $\m_4$ such that $R_4$ to $R_5$ is in the $x$-direction, 
$R_5$ to $R_6$ is in the $y$-direction, and $R_6$ to $R_7$ is in the $z$-direction.
With $(x_7, y_7, z_7)R_7$ the monomial regular system of parameters for $\m_7$, 
we define the local quadratic transform $R_7$ to $R_8$ in a manner similar to that from $R_3$ to $R_4$.
Thus $ x_8:=x_7$, $y_8 := \frac{y_7}{x_8} - 1$, and $z_8 := \frac{z_7}{x_7} - 1$.
If $V$ is a valuation birationally dominating $R_8$, then $v (y_4) = \frac{3}{2} v (x_4)$ and $v (z_4) = \frac{7}{4} v (x_4)$.

In a similar manner, for each positive integer $n$,   
we inductively define the sequence of local quadratic transforms
 $R_{4n} \subset R_{4n+1} \subset R_{4n+2} \subset R_{4n+3} \subset R_{4n+4}$.   
For each valuation domain $V$ that  birationally dominating $R_{4n+4}$, 
we have  $v (y_{4n}) = \frac{3}{2} v (x_{4n})$ and $v (z_{4n}) = \frac{7}{4} v (x_{4n})$.

Theorem~\ref{3.3a} implies that the sequence $\{R_i\}_{i \ge 0}$   switches strongly infinitely often.
Hence $V = \bigcup_{i \ge 0} R_i$ is a rank one valuation domain by Remark \ref{rank} and $V$ is a zero-dimension over $R$.
The rational rank of $V$ is $1$, for if the rational rank of $V$ were greater than $1$, then by Remark~\ref{0.2}.\ref{0.2.4}, 
there exists an integer $n > 0$ and elements $w_1, w_2 \in R_{4n}$ such that 
$\ord_{R_{4n}} (w_1) = \ord_{R_{4n}} (w_2) = 1$ and $v (w_1), v (w_2)$ are rationally independent.
Since $v (x_{4n}) < v (y_{4n}) < v (z_{4n}) < v (\m_{4n}^2)$ and $v (x_{4n}), v (y_{4n}), v (z_{4n})$ are rationally dependent, 
Remark~\ref{4.175} implies that $v (w_1)$ and $v (w_2)$ are in $\{ v (x_{4n}), v (y_{4n}), v (z_{4n}) \}$ and hence are rationally dependent.
\end{example}

Theorem~\ref{valideals} implies the following result for an infinite sequence
of monomial local quadratic transforms.

\begin{construction}\label{3.8}
Let notation  be as in Theorem~\ref{3.3} and assume that 
the sets $\frak D_{x}, ~\frak D_{y},~\ldots,~ \frak D_{z}$ are all infinite.  Then:
\begin{enumerate} 
\item  
 There exists an 
infinite  decreasing chain $\{I_n\}_{n\geq 0}$  of $v$-ideals in $R$
such that each $I_n$ is a  monomial ideal  and $\lambda_{R}(I_n/I_{n+1})=1$ for each $n \geq 0$.
\item We have $\bigcap_{n \ge 1}I_n = (0)$.
\item  If $p$ is a monomial in $R$, then $pV \cap R = I_n$ for some $n \ge 0$. 
\item  $V =\bigcup_{n\geq 0}R_n$ is a rank one valuation domain.
\end{enumerate}
\end{construction}

\begin{proof}  Since the sequence $\{R_n\}_{n \ge 0}$ is monomial, we have $R/\m = R_n/\m_n$
for all $n \ge 0$. Since the sets  
$\frak D_{x}, ~\frak D_{y},~\ldots,~ \frak D_{z}$ are all infinite, the sequence 
$\{R_n\}_{n \ge 0}$ switches infinitely often by Theorem~\ref{3.3}. Remark~\ref{rank} implies
that $V  =\bigcup_{n\geq 0}R_n$ is a rank one valuation domain. Theorem~\ref{valideals}
implies items~1, 2 and 3.
\end{proof}

\section{Properties of infinite sequences of local quadratic transforms}\label{c6}

\begin{setting}\label{6.1}
Let $(R, \m)$ be a $d$-dimensional regular local ring with $d \ge 2$ and   
let $V$ be a zero-dimensional real valuation domain birationally dominating $R$ with a corresponding valuation $v$.  
Consider  the  infinite  sequence  $\{(R_i, \m_i)\}_{i \ge 0}$   of local quadratic transforms  of $R$ along $V$,
	$$
R ~ = ~ R_0 \subset R_1 \subset \cdots \subset R_n \subset \cdots \subset  ~\cup_{n \ge 0}R_n ~ \subset~ V.
$$
Let $s$ denote the infinite sum,
	$$s = \sum_{i=0}^{\infty} v (\m_i).$$
  Thus $s$ is either $\infty$ or a positive real number.
\end{setting}

In Theorem~\ref{6.3},  we describe the value of $s$ in the monomial case.

\begin{theorem}\label{6.3}
Assume the notation of Setting~\ref{6.1} and that $\m = (x, y, \ldots, z)R$, where $v (x), v (y), \ldots,  v(z)$ are rationally independent.
Then,
\begin{enumerate}
\item $s ~=~ \sum_{i=0}^{\infty} v (\m_i)  ~ \le   ~\frac{v (x) + v (y) + \cdots + v (z)}{(d - 1)}$.  In particular, $s$ is finite.
\item   The following are equivalent:
\begin{enumerate}
\item Equality holds in Item~1. 
\item The sequence $\{(R_n, \m_n)\}_{n\ge 0}$  switches strongly infinitely often.
\item $\underset{n \to \infty}{\lim} v (w_n) = 0$, for each variable $w \in \Delta (\m_0)$.
\end{enumerate}
\end{enumerate}
\end{theorem}

\begin{proof}
For item~ 1,  we prove by induction that the following equation holds for every integer $n \ge 0$.
\begin{equation}\label{6.3.1}
\left( \sum_{i=0}^{n - 1} v (\m_i) \right) + \frac{v (x_n) + v (y_n) + \cdots + v (z_n)}{d-1} = \frac{v (x) + v (y) + \cdots + v (z)}{d-1}.
\end{equation}
Equality is clear in the case where $n = 0$.
Assume  that the claim is true for $n$.
By re-arranging variables, we may assume that $v (x_n) = v (\m_n)$, 
so that the local quadratic transform from $R_n$ to $R_{n+1}$ is in the $x$-direction.
Hence for every variable $w$ that is not $x$, $w_{n+1} := w_n / x_n$, so $v (w_{n+1}) = v (w_n) - v (x_n)$.
Thus,
\begin{align*}
	&\left( \sum_{i=0}^{n} v (\m_i) \right) + \frac{v (x_{n+1}) + v (y_{n+1}) + \cdots + v (z_{n+1})}{d-1} \\
	=& \left( \sum_{i=0}^{n - 1} v (\m_i) \right) + v (x_n) + \frac{( v (x_n) + v (y_n) + \cdots + v (z_n) ) - (d - 1) v (x_n)}{d - 1} \\
	=& \left( \sum_{i=0}^{n - 1} v (\m_i) \right) + \frac{v (x_n) + v (y_n) + \cdots + v (z_n)}{d - 1} \\
	=& \frac{v (x) + v (y) + \cdots + v (z)}{d-1}.
\end{align*}

By taking limits, we conclude that $s = \sum_{i=0}^{\infty} v (\m_i) \le \frac{v (x) + v (y) + \cdots + v (z)}{(d - 1)}$.

  For item~ 2, we let $E_n:=\sum_{i=0}^{n - 1} v (\m_i)$. We first show
$(a) \Rightarrow (c)$:
 By assumption, 
we have 
$$
\underset{n \to \infty}{\lim} E_n = \frac{v (x) + v (y) + \cdots + v (z)}{d-1} 
\quad \text{thus}   \quad  
 \underset{n \to \infty}{\lim} \frac{ v (x_n) + v (y_n) + \cdots + v (z_n)}{d-1} =0.
$$
For each  $w \in \Delta (\m_0)$,  the sequence $\{v(w_n)\}_{n\ge 0}$ is a
nonincreasing sequence of positive real numbers. Hence each 
$\underset{n \to \infty}{\lim} v (w_n)$ exists and is a nonnegative real number. 
Since $d \ge 2$, we have $\underset{n \to \infty}{\lim} \{ v (x_n) + v (y_n) + \cdots + v (z_n)\} =0,$
and thus  we
have 
$\underset{n \to \infty}{\lim} v (w_n) = 0$, for each variable $w \in \Delta (\m_0)$.

$(c) \Rightarrow (a)$: .
Assume  that $\underset{n \to \infty}{\lim} v (w_n) = 0$, 
for each variable $w \in \Delta (\m_0)$. That is, we have
$$
\underset{n \to \infty}{\lim} v (x_n)~ =~ 0, ~~\underset{n \to \infty}{\lim} v (y_n) ~= ~0,~\cdots,~\underset{n \to \infty}{\lim} v (z_n)~ =~ 0.
$$
 By Equation~\ref{6.3.1} we have
$$
E_n~=~\frac{v (x) + v (y) + \cdots + v (z)}{d-1}~-~\frac{v (x_n) + v (y_n) + \cdots + v (z_n)}{d - 1}.
$$
Hence
$$
\begin{aligned}
\underset{n \to \infty}{\lim} E_n &=~ \underset{n \to \infty}{\lim} \Big\{\frac{v (x) + v (y) + \cdots + v (z)}{d-1}\Big\} ~-~
                                    \underset{n \to \infty}{\lim} \Big\{ \frac{v (x_n) + v (y_n) + \cdots + v (z_n)}{d - 1}   \Big\}\\ 
                                  &=~\frac{v (x) + v (y) + \cdots + v (z)}{d-1}.
\end{aligned}
$$
The last equality follows by assumption. 

$(c) \Rightarrow(b)$:
Assume that $\underset{n \to \infty}{\lim} v (w_n) = 0$, 
for each variable $w \in \Delta (\m_0)$. 
Suppose by way of contradiction that the  sequence $\{(R_n, \m_n)\}_{n\ge 0}$ does not switch strongly infinitely often.
By Theorem~\ref{3.3}, there exists a variable $w \in \Delta (\m_0)$ such that $\frak D_w$ is finite.
That is, there exists positive integer $n_0$ such that the sequence 
$R_{n_0} \subset R_{n_0+1} \subset R_{n_0+2}\subset \cdots $ never goes in the $w$-direction. 
By replacing $n_0$ by zero, we may assume that 
the sequence  $\{(R_n, \m_n)\}_{n\ge 0}$ never goes in the $w$-direction.
That is, 
$$
v(w_n) \neq \min \{v(x_n), v(y_n), \ldots, v(w_n), \ldots, v(z_n)\}\quad \text{for every integer $n \ge 0$}.
$$
 From Equation~\ref{6.3.1}, we have
$$
\begin{aligned}
v(w)&=v(w_n) + E_n \quad \text{and}\\
v(x)+v(y)+\cdots+\widehat{v(w)}+\cdots +v(z)&=v(x_n)+v(y_n)+\cdots+\widehat{v(w_n)}+\cdots +v(z_n)+(d-2)E_n.
\end{aligned}
$$
By assumption, we have
$$
v(w)=\underset{n \to \infty}{\lim} E_n  \quad \text{and}\quad v(x)+v(y)+\cdots+\widehat{v(w)}+\cdots +v(z)=(d-2)\underset{n \to \infty}{\lim} E_n.
$$
Using both equalties, we have $v(x)+v(y)+\cdots+\widehat{v(w)}+\cdots +v(z)=(d-2)v(w)$, which is a contradiction,
since $v(x),v(y), \ldots, v(z)$ are rationally independent.

$(b) \Rightarrow (c)$:
By item (1), we have $s:=\sum_{i=0}^{\infty}v(\m_i)$ is finite.
Hence $\underset{n \to \infty}{\lim} v(\m_n)=\underset{n \to \infty}{\lim} E_{n+1}-\underset{n \to \infty}{\lim} E_n=0$.
Fix $w \in \Delta(\m_0)$, we have $v(w_n)=v(\m_n)$ for every $n$ such that the transform from $R_n$ to $R_{n+1}$ is in the $w$-direction.
By assumption, $\frak D_w=\infty$, and hence $\underset{n \to \infty}{\lim} v(w_n) = \underset{n \to \infty}{\lim} v(\m_n) = 0$
\end{proof}

\begin{proposition}\label{6.2}
Assume the notation of Setting~\ref{6.1}. Then we have the following:

\begin{enumerate}
\item If $V$ is a DVR, then $s = \infty$, and $\cup_{n \ge 0}R_n = V$.
\item If $s = \infty$, then the sequence $\{R_n\}_{n \ge 0}$  switches strongly infinitely often. 
\item If $V$ has rational rank $r \ge 2$, then $s < \infty$. Thus if $s = \infty$, 
then $V$ has rational rank $r = 1$.
\item If $V$ has rational rank $1$ and $V$ is not a DVR, then both $s = \infty$ and $s < \infty$ are possible.
\end{enumerate}
\end{proposition}

\begin{proof}
If $V$ is discrete, we may assume the value group of $V$ is 
$\Z$, and hence  $v (\m_n) \ge 1$ for all $n \ge 0$. Thus item~1 is clear.

Item~2 is proved by Granja, Martinez and Rodriguez in \cite[Proposition~23]{GMR}.

To see Item~3, assume that $V$ has rational rank $r \ge 2$ and suppose by way of contradiction that $s = \infty$.
By Item~2, the sequence switches strongly infinitely often.
By Remark~\ref{0.2}.\ref{0.2.4}, there is some $n > 0$ and 
elements $x, y$ that are part of some regular system of parameters for $R_n$ such that $v (x), v (y)$ are rationally independent.
By replacing $R$ by $R_n$, we may assume that there are elements $x, y$ in some regular system of parameters for $R$
 such that $v (x), v (y)$ are rationally independent.

We  show that $s \le v (x) + v (y)$ by inductively proving that for all $n \ge 0$, 
there are elements $x_n, y_n$ of some regular system of parameters for $R_n$
 such that $v(x_n), v(y_n)$ are rationally independent and $\left( \sum_{i=0}^{n - 1} v (\m_i) \right) + v (x_n) + v (y_n) \le v (x) + v (y)$.

Taking  $x_0 = x$ and $y_0 = y$, the base case $n = 0$ is clear. 

Assume  the claim is true for $n$. Thus  we have elements $x_n, y_n \in \m_n$ such that 
$v(x_n), v(y_n)$ are rationally independent and 
\begin{equation} \label{6.2.1}
\left( \sum_{i=0}^{n - 1} v (\m_i) \right) + v (x_n) + v (y_n) ~ \le ~ v (x) + v (y).
\end{equation}
Let $z \in m_n$ denote an element of minimal $v$-value.
Then $v (z)$ and at least one of $v (x_n), v (y_n)$ are rationally independent, 
so we may assume without loss of generality that $v (z), v (y_n)$ are rationally independent.
Thus $z, y_n$ are part of a regular system of parameters for $R_n$.
Set $x_{n+1} = z$ and $y_{n+1} = \frac{y_n}{z}$, so $v (x_{n+1}), v (y_{n+1})$ are rationally independent, and $v (y_{n+1}) = v (y_n) - v (z)$.
We have $v (\m_n) = v (z) \le v (x_n)$.
Thus,
\begin{align*}
	\left( \sum_{i=0}^{n} v (\m_i) \right) + v (x_{n+1}) + v (y_{n+1})
	&=  \left( \sum_{i=0}^{n - 1} v (\m_i) \right) + v (z) + v (z) + v (y_n) - v (z) \\
	&\le \left( \sum_{i=0}^{n - 1} v (\m_i) \right) + v (x_n) + v (y_n) \\
	&\le v (x) + v (y)
\end{align*}
where the last inequality follows from Equation~\ref{6.2.1}.

We conclude that 
 $s < \infty$, in contradiction to the assumption that $s = \infty$.
This proves Item~3.

Example~\ref{4.18} shows that $s < \infty$ is possible in the case  where the rational rank of $V$ is $1$ and $V$ is not a DVR.
The pattern gives $s = (1 + \frac{1}{2} + \frac{1}{4} + \frac{1}{4}) + (\frac{1}{4} + \frac{1}{8} + \frac{1}{16} + \frac{1}{16}) + \ldots$, 
so $s = \sum_{n=0}^{\infty} \frac{1}{2^n} + 2 \sum_{n=1}^{\infty} \frac{1}{4^n} = \frac{8}{3}$.
To complete the proof of item~4, we show in Examples~\ref{7.13} and \ref{7.14} that 
$s = \infty$ is possible in the case  where the rational rank of 
$T=\bigcup_{n \ge 0}A_n$ is $1$ and $T$ is not a DVR.
\end{proof}

\begin{setting} \label{htonesetting}
Let $(R, \m)$ be a $d$-dimensional regular local ring with $d \ge 2$,
and  with $R := R_0$, let $\{(R_n,\m_n)\}_{n \ge 0}$ be an infinite sequence 
of local quadratic transforms with $\dim R_n = \dim R$ for all $n$.
Let $y_n \in \Delta(\m_n)$ be such that the transform from $R_n$ to $R_{n+1}$ is
in the $y_n$-direction.
Let $S:=\bigcup_{n \ge 0}R_n$. It is well known that $S$ is a normal local domain 
with unique maximal 
ideal $\m_S:=\bigcup_{n\ge 0}\m_n$.
Assume that for some
nonnegative integer $j$ there exists a regular prime element $x$ in $R_j$
such that $S \subset (R_j)_{xR_j}$.  In considering properties of the directed union,
we may assume the sequence starts at $R_j$. Thus with a change of notation,  we 
assume that $j = 0$.  Let $P:=xR_{xR} \cap S$ and let $ T:=S /P$ and $\m_T:=\m_S /P$.
For each $n \ge 0$, let $x_n$ be the transform of $x$ in $R_n$.
Then $P=\bigcup_{n\ge 0}x_n R_n$. For each $n \ge 0$, let
$$
A_n:=\frac{R_n}{x_n R_n} \quad \text{and} \quad \n_n:=\frac{\m_n}{x_n R_n}.
$$
Each $(A_n, \n_n)$ is a $(d-1)$-dimensional regular local ring. Moreover 
$\{(A_n, \n_n)\}_{n\ge 0}$ is an infinite sequence of local quadratic transforms.
Hence $T=\bigcup_{n \ge 0}A_n$ is a normal  local domain with 
maximal ideal $\m_T=\bigcup_{n \ge 0}\n_n$.
Let $\nu_x$ be the $x$-adic valuation associated with the valuation domain $R_{xR}$.
 For an element $g \in S$, let $\overline{g}$ denote  the image of $g$ in $T$.
\end{setting} 

\begin{theorem} \label{htonedirected}
Let notation be as in Setting~\ref{htonesetting}. The following are equivalent.
\begin{enumerate}
\item  The sequence $\{(R_n, \m_n)\}_{n \ge 0}$ is height one directed. 
\item  $S$ is a valuation domain.
\item $T$ is a rank one valuation domain with associated real valuation $\nu$ 
such that $\sum_{n \ge 0} \nu(\n_n) = \infty$.
\item  The sequence $\{(A_n, \n_n)\}_{n \ge 0}$ switches strongly infinitely often,
and for the associated real valuation $\nu$ we have $\sum_{n \ge 0} \nu(\n_n) = \infty$.
\end{enumerate}
\end{theorem} 

\begin{proof}
$(1) \Rightarrow (2)$: This is proved by Granja in \cite[Theorem~13]{Gr}.

$(2) \Rightarrow (1)$: By \cite[Proposition~7]{Gr}, the valuation domain $S$ has 
rank  either $1$ or  $2$.
By assumption, the sequence $\{(R_n, \m_n)\}_{n \ge 0}$ does not switch strongly infinitely often.
Hence by \cite[Theorem~13]{Gr},  the sequence $\{(R_n, \m_n)\}_{n \ge 0}$ is 
height one directed and $S$ has rank $2$.

$(2) \Rightarrow (3)$: Assume that $S$ is a valuation domain. Then $T = S/P$ is a
valuation domain of the field $R_{xR}/xR_{xR}$.
 Since $P$ is a nonzero nonmaximal prime ideal in $S$
and $\dim S=2$,  we have 
$\dim T=1$.
It remains to show that $s:=\sum_{n \ge 0} \nu(\n_n) = \infty$.
Suppose by way of contradiction that $s < \infty$. Let $z \in \Delta(\m_0)$ be such that $\nu_x(z)=0$, (that is $zR \neq xR$).
By the Archimedean property, there exists an integer $N > 0$ such that $N \nu(\overline{z}) > s$.
Consider the following element
$$
f:=x - z^N.
$$
For simplicity of notation, we set $g:=z^N$. Then $\nu(\overline{g})=\nu(\overline{z^N})>s$.
For each $n \ge 0$, let $f_n$ be the transform $f$ in $R_n$. 
Let $\nu_f$ be the $f$-adic valuation associated with the valuation domain $R_{fR}$.
To see $S \subset R_{fR}$, we  prove by induction on $n$ the following claim;
\begin{claim}\label{htoneclaim}
For every integer $n \ge 0$, there exists $g_n \in R_n$ such that 
$$
f_n=x_n-g_n, \quad \text {where $~\nu_f(f_n) = 1, \quad  \nu_f(g_n)=0~$ and $~\nu(\overline{g_n}) > \sum_{i=n}^{\infty}\nu(\n_i)$.}
$$
\end{claim}

\begin{proof}
 The case where $n=0$ is clear by construction. Assume that the claim holds for $n$.
Since   the transform from $R_n$ to $R_{n+1}$ is in the $y_n$-direction, we have 
$$
f_{n+1}=\frac{f_n}{y_n}=\frac{x_n}{y_n}-\frac{g_n}{y_n}=x_{n+1}-\frac{g_n}{y_n}.
$$
We set  $g_{n+1}:=\frac{g_n}{y_n}$. Then we have 
$$
\begin{aligned}
&\nu_f(g_{n+1})=\nu_f(g_n)-\nu_f(y_n)=0 \quad \text{and}\\
&\nu(\overline{g_{n+1}}) =\nu\Big(\overline{\frac{g_n}{y_n}}\Big)=\nu(\overline{g_n})-\nu(\overline{y_n})
               > \sum_{i=n}^{\infty}\nu(\n_i)-\nu(\n_n)=\sum_{i=n+1}^{\infty}\nu(\n_i).
\end{aligned}
$$
\end{proof}

By Claim~\ref{htoneclaim}, we have $R_n \subset R_{fR}$ for each positive integer $n$. 
Hence $S \subset R_{fR}$. Since $R_{xR} \ne R_{fR}$, this contradicts the fact  that 
the sequence $\{(R_n, \m_n)\}_{n \ge 0}$ is height one directed.

$(3) \Rightarrow (2)$: Assume that $T$ is a rank one valuation domain with associated 
real valuation $\nu$ 
such that $\sum_{n \ge 0} \nu(\n_n)  = \infty$.
Let $W$ be the composite valuation domain defined by the valuations $\nu$ and $\nu_x$.
Thus 
$$
W ~= ~ \{ \alpha \in \mathcal Q(R) ~|~ \nu_x(\alpha) > 0 \quad \text{or} \quad \nu_x(\alpha)=0 \quad \text{and}\quad \nu(\widetilde{\alpha}) \ge 0\},
$$
 where $\widetilde{\alpha}$ denotes  the image of $\alpha$ in $R_{xR}/xR_{xR}$.
We have $S \subset W \subset R_{xR}$.
We prove the following claim : 
\begin{claim}\label{compval}
$W \subset S$. 
\end{claim}

\begin{proof}
Let $\alpha \in W$.  Then $\alpha = x^t \frac{g}{h}$ ,where $ g, h \in R \setminus xR$
are relatively prime in $R$ 
and $t = \nu_x(\alpha) \ge 0$. We consider two cases\\
 (Case i): Assume that $t = \nu_x(\alpha)=0$. Since $\alpha \in W$, we have 
$\nu(\widetilde{\alpha}) \ge 0$.  Let 
$$
r_0 ~:= ~ \min \{~\ord_{R_0}(g),~ \ord_{R_0}(h)~\}.
$$
If $r_0 = 0$, then at least one of $g$ or $h$ is a unit in $R$. If $h$ is a unit in $R$
then $\alpha \in R \subset S$. If $g$ is a unit in $R$, then $\nu(\widetilde g) = 0$. 
Since $\nu(\widetilde \alpha) = \nu(\widetilde g) - \nu(\widetilde h) \ge 0$, we must also have 
$\nu(\widetilde h) = 0$, and $\alpha \in R$. Assume that $r_0 > 0$ and set $g_0 := g$
and $h_0 := h$.  
Since  the transform from $R = R_0$ to $R_1$ is
in the $y = y_0$-direction 
 there exist elements $g_1$ and $h_1$ in $R_1$ such that $g = y^{r_0}g_1$
and $h = y^{r_0}h_1$.  Then $\alpha=\frac{g_1}{h_1}$. 
Notice that the ideal $(g_1, h_1)R_1$ is the transform in
$R_1$ of the ideal $(g, h)R$.

Let $r_1:=\min \{~\ord_{R_1}(g_1),~ \ord_{R_1}(h_1)~\}$. 
If $r_1 = 0$, then $\alpha \in R_1$. If $r_1 > 0$, then since  
the transform from $R_1$ to $R_2$ is
in the $y_1$-direction 
there exist elements $g_2$ and $h_2$ in $R_2$ such that $g_1 = y_1^{r_1}g_2$
and $h_1 = y_1^{r_1}h_2$.  Then $\alpha=\frac{g_1}{h_1}=\frac{g_2}{h_2}$.
The ideal $(g_2, h_2)R_2$  is the transform in $R_2$ of
the ideal $(g_1, h_1)R_1$ and also the transform of the ideal $(g, h)R$.

Suppose the transform of  $(g, h)R$ in $R_n$ is a proper ideal in $R_n$ 
for every integer $n \ge 0$. 
Then $r_n:=\min \{~\ord_{R_n}(g_n),~ \ord_{R_n}(h_n)~\}$ is positive for all $n \ge 0$.
This gives infinite sequences $\{g_{n+1}\}_{n \ge 0}, \{h_{n+1}\}_{n \ge 0}$ such that
$g_{n+1},~ h_{n+1} \in R_{n+1} \setminus x_{n+1} R_{n+1}$, where $ g_{n} = y_n^{r_n}g_{n+1}$
and $h_{n} = y_n^{r_n}h_{n+1}$. 
Then for every  integer $n \ge 0$ we have
$$
\begin{aligned}
&\nu(\widetilde{g_{n+1}})=\nu(\widetilde{g_{n}})-r_n \nu(\widetilde{y_n})=\nu(\widetilde{g_{n}})-r_n \nu(\n_n) \leq \nu(\widetilde{g_{n}})- \nu(\n_n),\\
&\nu(\widetilde{h_{n+1}})=\nu(\widetilde{h_{n}})-r_n \nu(\widetilde{y_n})=\nu(\widetilde{h_{n}})-r_n \nu(\n_n) \leq \nu(\widetilde{h_{n}})- \nu(\n_n).
\end{aligned} 
$$
Hence for every  integer $n \ge 1$ we have
$$
\nu(\widetilde{g_{n}}) \leq  \nu(\widetilde{g_{0}})- \sum_{i=0}^{n-1} \nu(\n_i) \quad 
\text{and}\quad \nu(\widetilde{h_{n}}) \leq  \nu(\widetilde{h_{0}})- \sum_{i=0}^{n-1} \nu(\n_i). 
$$
By taking limits, we have 
$$
\underset{n \to \infty}{\lim} \nu(\widetilde{g_{n}}) \leq  \nu(\widetilde{g_{0}})- \sum_{i=0}^{\infty} \nu(\n_i) \quad 
\text{and}\quad \underset{n \to \infty}{\lim} \nu(\widetilde{h_{n}}) \leq  \nu(\widetilde{h_{0}})- \sum_{i=0}^{\infty} \nu(\n_i)
$$
Since $\sum_{i=0}^{\infty} \nu(\n_i)=\infty$, this is a contradiction.
Hence   $r_n=\min \{~\ord_{R_n}(g_n),~ \ord_{R_n}(h_n)~\}=0$ for some integer $n \ge 0$. 
Then either $g_n$ is a unit in $R_n$ or $h_n$ is a unit in $R_n$.
If $h_n$ is a unit in $R_n$, then $\alpha \in R_n$, and hence $\alpha \in S$.
If $g_n$ is a unit in $R_n$, then $\nu(\widetilde{g_n}) = 0$.  
Since $\nu(\widetilde \alpha) = \nu(\widetilde{g_n}) - \nu(\widetilde{h_n}) \ge 0$, we must also have 
$\nu(\widetilde{h_n}) = 0$, and $\alpha \in R$.

(Case ii): Assume that $t=\nu_x(\alpha) > 0$. Let $\beta_0:= \frac{g}{h}$.  
If $\nu(\widetilde \beta_0) \ge 0$, then by Case~i, $\beta_0 \in S$    
and hence $x^t\beta_0 = \alpha \in S$. If $\nu(\widetilde \beta_0) < 0$, let $\beta_1 := y_0^t\beta_0$.
Then 
$$ 
\alpha~ = ~ x^t\beta_0 ~ = ~ x_1^ty^t\beta_0 ~ =  ~x_1^t\beta_1.
$$
If $\nu(\widetilde \beta_1) \ge 0$, then by (Case~i), $\beta_1 \in S$    
and hence $x_1^t\beta_1 = \alpha \in S$. If $\nu(\widetilde \beta_1) < 0$, let $\beta_2 := y_1^t\beta_1$.
Then 
$$ 
\alpha~ = ~ x_1^t\beta_1 ~ = ~ x_2^ty_1^t\beta_1~=  ~x_2^t\beta_2 \quad 
\text{ and } \quad \beta_2 ~=~y_1^ty_0^t\beta_0 
$$
For each positive integer $n$, we define $\beta_{n+1} = y_n^t\beta_n$. It follows that
$$
\alpha~ = ~ x_n^t\beta_n ~ = ~ x_{n+1}^ty_n^t\beta_n~=  ~x_{n+1}^t\beta_{n+1} \quad 
\text{ and } \quad 
\beta_{n+1} ~=~ (y_ny_{n-1} \cdots y_0)^t\beta_0.
$$
Thus we have 
$$
\nu(\widetilde {\beta_{n+1}}) =~t\Big(\sum_{i=0}^{n} \nu(\widetilde {y_i})\Big) + 
\nu(\widetilde {\beta_0)}~=~ ~
t\Big(\sum_{i=0}^{n} \nu(\n_i)\Big) + \nu(\widetilde {\beta_0}).
$$
Since $\sum_{i=0}^{\infty} \nu(n_i)=\infty$, we have 
$\nu(\widetilde {\beta_{n}}) \ge  0$ for all sufficiently 
large integers $n$. 
Then by Case~i, we have $\beta_{n} \in S$    
and hence $x_{n}^t\beta_{n} = \alpha \in S$.
This completes the proof of Claim~\ref{compval}.
\end{proof}

Hence by Claim~\ref{compval} we have $S=W$, and we conclude that $S$ is a valuation domain.

$(3) \Leftrightarrow (4)$: This equivalence follows from Shannon(\cite[Proposition~4.18]{S}) and Granja(\cite[Theorem~13]{Gr}).
\end{proof}

\begin{remark}\label{7.10}
Let $(R_0,\m_0)$ be a $d$-dimensional regular local ring with $d \ge 2$,
and  let $\{(R_n,\m_n)\}_{n \ge 0}$ be an infinite sequence 
of local quadratic transforms with $\dim R_n = \dim R_0$ for all $n$.
Assume  that $S:=\bigcup_{n \ge 0}R_n$ is a valuation domain of rank greater than one.
If $R_0$ is excellent and equicharacteristic zero,
 Granja proves in \cite[Theorem~17]{Gr} that there exists for some
nonnegative integer $j$  a regular prime  $x$ in $R_j$
such that $S \subset (R_j)_{xR_j}$. 
\end{remark}

\begin{remark}\label{7.11}
Let notation be as in Setting~\ref{htonesetting}.
Assume that $S$ is a valuation domain. 
Let $G$ be the value group of a valuation $v$ associated with $S$.
and let $H$ be the value group of  a valuation $\nu$ associated with $T$.
Then  we have the following :
\begin{enumerate}
\item  $S$ has  rational rank $2$ and $G \cong \Bbb Z \oplus H$.
\item  $T$ has  rational rank $1$.
\item If $T$ is  DVR,  then $G \cong \Bbb Z^2$.
\end{enumerate}
\end{remark}

\begin{proof}
To see item~1, by Granja \cite[Proposition~14]{Gr}, the valuation domain $S$ has rational
rank  $2$.
By \cite[Theorem~15, page~40]{ZS2} 
the set  $G_P:=\{~\pm v(\alpha) ~|~ \alpha \in S \setminus P~\}$ is an isolated subgroup of $G$,
(that is, $G_P$ is a segment and a proper subgroup of $G$).
By \cite[Theorem~17, page~43]{ZS2} the group $H$ and the group $G_P$ 
are isomorphic  as  ordered groups.  Since $S_P=R_{xR}$ is a DVR,
the value group of $S_P$ is $\Bbb Z$. Hence by \cite[Theorem~17, page~43]{ZS2}
the groups $G/G_P$ and $\Bbb Z$ are  order isomorphic. It follows that 
$G \cong \Bbb Z \oplus G_P \cong \Bbb Z \oplus H$.
Items~2 and 3 follow from item~1.
\end{proof}

\begin{remark}\label{7.12}
Let notation be as in Setting~\ref{htonesetting}.
\begin{enumerate}
\item If $d=3$,  then $T=\bigcup_{n \ge 0}A_n$ is a valuation domain by \cite[Lemma~12]{A}. 
\item Let $T$ be as in Setting~\ref{htonesetting}.
If $T$ is a rank one valution domain and $s=\sum_{i\ge 0}^{\infty} \nu(\n_i) < \infty$,
then there exist  infinitely many choices for the positive integer $N$  and hence 
infinitely many nonassociate regular prime elements $f = x - z^N$ in $R$  such that 
$S \subset R_{fR}$. Let $Q = S \cap fR_{fR}$. Then $S/Q$ is an infinite directed union
of $d-1$-dimensional regular local domains. If $d = 3$, then by item~1, 
each of the infinitely many  $S/Q$ is a 
valuation domain. 
\end{enumerate}
\end{remark}

In Example~\ref{7.13} we construct an infinite directed sequence $\{(R_n,\m_n)\}_{n \ge 0}$ 
of local quadratic transforms of a $3$-dimensional regular local ring $R = R_0$
such that $S = \bigcup_{n \ge 0}R_n$ is a rank 2 valuation domain with value group 
$\mathbb Z \oplus H$, where $H$ is rational rank 1 but not discrete.

\begin{example}\label{7.13}
Let $(R_0, \m_0)$ be a $3$-dimensional regular local ring, and let $\m_0 = (x, y, z) R_0$. 
We define a sequence $\{(R_n, \m_n)\}_{n\ge 0}$ of local quadratic transforms of $R_0$ as follows.
The sequence from $R_0$ to $R_3$ is
$$
R:=R_0 ~\overset{y}  {\subset}~  R_1 ~\overset{z}{\subset}~ R_2 ~\overset{y'}{\subset}~ R_3
$$ 
defined as
$$ 
\begin{aligned} 
R_1 &= R_0[\frac{\m}{y}]_{(\frac{x}{y}, y, \frac{z}{y})} \quad\text{ and } \quad \m_1 = (x_1, y_1, z_1)R_1 \\
R_2 &= R_1[\frac{\m_1}{z_1}]_{(\frac{x_1}{z_1}, \frac{y_1}{z_1}, z_1)} \quad\text{ and } \quad \m_2 = (x_2, y_2, z_2)R_2 
=(\frac{x_1}{z_1}, \frac{y_1}{z_1}, z_1)R_2  \\
R_3 &= R_2[\frac{\m_2}{y_2}]_{(\frac{x_2}{y_2}, y_2, \frac{z_2}{y_2} - 1)} \quad\text{ and } \quad \m_3 = (x_3, y_3, z_3)R_3 
= (\frac{x_2}{y_2}, y_2, \frac{z_2}{y_2}-1)R_3 
\end{aligned}
$$
Starting from $(R_3, \m_3)$ and $\m_3 = (x_3, y_3, z_3)R_3$, we define a sequence
$$
R_3 ~\overset{y}  {\subset}~  R_4 ~\overset{y}  {\subset}~  R_5 ~\overset{z}{\subset}~ R_6 ~\overset{y'}{\subset}~ R_7
$$ 
of local quadratic transform with respect to the regular system of parameters $x_3, y_3, z_3$ of $\m_3$ 
such that the $2^1$ transforms from $R_3$ to $R_5$ are both  monomial in the $y$-direction, 
          the transform from $R_5$ to $R_{6}$ is monomial in the $z$-direction, and 
   the transform  from $R_{6}$ to $R_{7}$ is defined in a manner similar to that from $R_2$ to $R_3$.
Thus we have
$$
\m_{7}=(x_{7}, y_{7}, z_{7})R_{7}=\Big(\frac{x_6}{y_6}, y_6, \frac{z_{6}}{y_{6}}-1\Big)R_{7}
=\Big(\frac{x_3}{y_3^3}, \frac{y_3^3}{z_3}, \frac{z_3^2}{y_3^5}-1\Big)R_{7}.
$$
Let $t_0:=0$ and  $t_1:=3$. For each integer $n \ge 2$,  let $t_n=2 t_{n-1}-t_{n-2}+2^{n-2}$.
We inductively define a sequence of local quadratic transforms
with respect to the  regular system of parameters  $x_{t_n}, y_{t_n}, z_{t_n}$ of  $\m_{t_n}$
as follows:

$$
R_{t_n}\overbrace{~\overset{y}  {\subset}~R_{t_n+1}~\overset{y}  {\subset}~\cdots ~ \overset{y}  {\subset}}^{\text{$2^n$ times}}~ R_{t_n+2^n} 
~\overset{z}  {\subset}~ R_{t_n+2^n+1} ~\overset{y'}  {\subset}~  R_{t_n+2^n+2} ,
$$ 
The $2^n$ transforms from $R_{t_n}$ to $R_{t_n+2^n}$ are all monomial in the $y$-direction, 
 the transform from $R_{t_n+2^n}$ to $R_{t_n+2^n+1}$ is monomial in the $z$-direction,
and  the transform from $R_{t_n+2^n+1}$ to $R_{t_n+2^n+2}$ 
is defined in a manner similar to that from $R_6$ to $R_7$.
Thus we have
$$
\begin{aligned}
\m_{t_n}=(x_{t_n}, y_{t_n}, z_{t_n})R_{t_n}&\overbrace{~\overset{y}{\to}~\cdots~\overset{y}  {\to}}^{\text{${2^n}$ times}}
\m_{t_n+{2^n}}=\Big(\frac{x_{t_n}}{y_{t_n}^{2^n}},~ y_{t_n},~ \frac{z_{t_n}}{y_{t_n}^{2^n}}\Big)R_{t_n+{2^n}}\\
&~\overset{z}{\to}~\m_{t_n+{2^n}+1}=\Big(\frac{x_{t_n}}{z_{t_n}},~ \frac{y_{t_n}^{2^n+1}}{z_{t_n}},~ \frac{z_{t_n}}{y_{t_n}^{2^n}}\Big)R_{t_n+{2^n}+1}\\
&~\overset{y'}{\to}~\m_{t_n+{2^n}+2}=\Big(\frac{x_{t_n}}{y_{t_n}^{2^n+1}},~ \frac{y_{t_n}^{2^n+1}}{z_{t_n}},~ 
\frac{z_{t_n}^2}{y_{t_n}^{2(2^n)+1}}-1\Big)R_{t_n+{2^n}+2}
\end{aligned}
$$
Let $S:=\bigcup_{n \ge 0}R_n$ and $\m_S:=\bigcup_{n \ge 0}\m_n$. 
Since we never divide in the $x$-direction, we have $S ~\subset~ R_{xR}$. 
Let $P:=xR_{xR} \cap S$ and let $T:=S/P$ and $\m_T:=\m_S/P$.
For each $n \ge 0$, let $x_n$ be the transform of $x$ in $R_n$.
Then $P~=~\bigcup_{n\ge 0} x_n R_n$.
For each $n \ge 0$, let
$$
A_n:=\frac{R_n}{x_n R_n} \quad \text{and} \quad \n_n:=\frac{\m_n}{x_n R_n}.
$$
Each $(A_n, \n_n)$ is a $2$-dimensional regular local ring with maximal ideal $\n_n$ generated
by the images of $y_n$ and $z_n$ in $A_n$. By identifying $y_n$ and $z_n$ with their images 
in $A_n$, we have $\n_n = (y_n, z_n)A_n$. 
 Moreover 
$\{(A_n, \n_n)\}_{n\ge 0}$ is an infinite sequence of local quadratic transforms of $(A_0, \n_0)$, 
where $\n_0:=(y, z)A_0$. By \cite[Lemma~12]{A}, the ring  $T=\bigcup_{n \ge 0}A_n$ 
is a valuation domain with maximal ideal $\m_T=\bigcup_{n \ge 0}\n_n$.
Let $\nu$ be a valuation  associated with the valuation domain $T$.
The sequence $\{(A_n, \n_n)\}_{n\ge 0}$ is determined by the sequence $\{(R_n, \m_n)\}_{n \ge 0}$.
With the integers $t_n$ as defined in the construction of the 
sequence $\{(R_n,\m_n) \}_{n \ge 0}$, we have 
$$
A_{t_n}\overbrace{~\overset{y}  {\subset}~A_{t_n+1}~\overset{y}  {\subset}~\cdots ~ \overset{y}  {\subset}}^{\text{$2^n$ times}}~ A_{t_n+2^n} 
~\overset{z}  {\subset}~ A_{t_n+2^n+1} ~\overset{y'}  {\subset}~  A_{t_n+2^n+2} ,
$$ 
The $2^n$  transforms from $A_{t_n}$ to $A_{t_n+2^n}$ are all monomial in the $y$-direction, 
the transform from $A_{t_n+2^n}$ to $A_{t_n+2^n+1}$ is monomial in the $z$-direction, and we have
$$
\begin{aligned}
\n_{t_n}=( y_{t_n}, z_{t_n})A_{t_n}&\overbrace{~\overset{y}{\to}~\cdots~\overset{y}  {\to}}^{\text{${2^n}$ times}}
\n_{t_n+{2^n}}=\Big( y_{t_n},~ \frac{z_{t_n}}{y_{t_n}^{2^n}}\Big)A_{t_n+{2^n}}\\
&~\overset{z}{\to}~\n_{t_n+{2^n}+1}=\Big( \frac{y_{t_n}^{2^n+1}}{z_{t_n}},~ \frac{z_{t_n}}{y_{t_n}^{2^n}}\Big)A_{t_n+{2^n}+1}\\
&~\overset{y'}{\to}~\n_{t_n+{2^n}+2}=\Big( \frac{y_{t_n}^{2^n+1}}{z_{t_n}},~ \frac{z_{t_n}^2}{y_{t_n}^{2(2^n)+1}}-1\Big)A_{t_n+{2^n}+2}
\end{aligned}
$$
Since $\frac{z_{t_n+2^n+1}}{y_{t_n+2^n+1}}=\frac{z_{t_n}^2}{y_{t_n}^{2(2^n)+1}} \in A_{t_n+2^n+2} \setminus \n_{t_n+2^n+2}$, 
we have $ \nu(z_{t_n})=\frac{2(2^n)+1}{2} \nu(y_{t_n})$. 

Assume that $\nu(y)=1$. Then we have:
$$
\begin{aligned}  
s:= \sum_{i=0}^{\infty} \nu (\n_i)&=1 +\frac{1}{2}+\frac{1}{2}+\{\frac{1}{2}+\frac{1}{2}\}
+\frac{1}{4}+\frac{1}{4}+\{\frac{1}{4}+\frac{1}{4}+\frac{1}{4}+\frac{1}{4}\}\\
&+\frac{1}{8}+\frac{1}{8}+\overbrace{\{\frac{1}{8}+\frac{1}{8}+\cdots+\frac{1}{8}\}}^{\text{$2^3$ times}}
+\frac{1}{16}+\frac{1}{16}+\overbrace{\{\frac{1}{16}+\frac{1}{16}+\cdots+\frac{1}{16}\}}^{\text{$2^4$ times}}+\cdots =\infty.
\end{aligned}
$$
By Theorem~\ref{htonedirected}, the ring $S$ is the valuation domain,  
the sequence $\{(R_n, \m_n)\}_{n \ge 0}$ is height one directed,
and $S$ has rank  $2$. Let $G$ be the value group of $S$ and let  $H$ be the value group of $T$. 
By Remark~\ref{7.11},  $S$ has rational rank $2$
and $G \cong \Z \oplus H$.
Clearly, $T$ has  rational rank $1$, and $T$ is not a DVR.
\end{example}

In Example~\ref{7.14} we construct an example similar to that of Example~\ref{7.13},
but with $\dim R = 4$ and $\dim A = 3$.

\begin{example}\label{7.14}
Let $(R_0, \m_0)$ be a $4$-dimensional regular local ring, and let $\m_0 = (x, y, z, w) R_0$. 
We define a sequence $\{(R_n, \m_n)\}_{n \ge 0}$ of local quadratic transforms of $R_0$ as follows:
The sequence from $R_0$ to $R_5$ is
$$
R:=R_0 ~\overset{y}  {\subset}~  R_1 ~\overset{y}{\subset}~ R_2 ~\overset{z}{\subset}~ R_3~\overset{w}{\subset}~ R_4~\overset{y'}{\subset}~ R_5
$$ 
defined by
$$ 
\begin{aligned} 
R_1 &= R_0[\frac{\m}{y}]_{( \frac{x}{y}, y,  \frac{z}{y}, \frac{w}{y})} \quad\text{ and } \quad \m_1 = (x_1, y_1,z_1, w_1)R_1        \\
R_2 &= R_1[\frac{\m_1}{y_1}]_{(\frac{x_1}{y_1}, y_1, \frac{z_1}{y_1}, \frac{w_1}{y_1})} \quad\text{ and } 
\quad \m_2 = (x_2, y_2, z_2, w_2)R_2   \\
R_3 &= R_2[\frac{\m_2}{z_2}]_{(\frac{x_2}{z_2},\frac{y_2}{z_2}, z_2, \frac{w_2}{z_2} )} \quad\text{ and } 
\quad \m_3 = (x_3, y_3, z_3, w_3)R_3 \\
R_4 &= R_3[\frac{\m_3}{w_3}]_{(\frac{x_3}{w_3}, \frac{y_3}{w_3}, \frac{z_3}{w_3}, w_3)} \quad\text{ and } 
\quad \m_4 = (x_4, y_4, z_4, w_4)R_4 \\
R_5 &= R_4[\frac{\m_4}{y_4}]_{(\frac{x_4}{y_4}, y_4, \frac{z_4}{y_4}-1,\frac{w_4}{y_4}-1)} \quad\text{ and } 
\quad 
\m_5 = (x_5, y_5, z_5, w_5)R_5\\ \qquad   \qquad  =&  \Big(\frac{x_4}{y_4}, y_4, \frac{z_4}{y_4}-1,\frac{w_4}{y_4}-1 \Big)R_5
                          =\Big(\frac{x}{y^3},\frac{y^3}{w}, \frac{z^2}{y^5}-1, \frac{w^2}{y^3z}-1\Big)R_5
\end{aligned}
$$
Starting from $(R_5, \m_5)$ and $\m_5 = (x_5, y_5, z_5, w_5)R_5$, 
we define a sequence
$$
R_5 ~\overset{y}{\subset}~  R_6 ~\overset{y}{\subset}~R_7 ~\overset{y}{\subset}~  R_8 ~\overset{y}{\subset}~  R_9
~\overset{z}{\subset}~  R_{10} ~\overset{w}{\subset}~ R_{11} ~\overset{y'}{\subset}~ R_{12}
$$ 
of local quadratic transform with respect to the regular system of parameters $x_5, y_5, z_5, w_5$ of $\m_5$ 
such that the $2^2$ transforms from $R_5$ to $R_9$ are all monomial in the $y$-direction, 
          the transform from $R_9$ to $R_{10}$ is monomial in the $z$-direction,
           the transform from $R_{10}$ to $R_{11}$ is monomial in the $w$-direction, and 
    the transform  from $R_{11}$ to $R_{12}$ is defined  in a manner similar to that from $R_4$ to $R_5$
Thus we have
$$
\m_{12}=(x_{12}, y_{12}, z_{12}, w_{12})R_{12}=\Big(\frac{x_{11}}{y_{11}}, y_{11}, \frac{z_{11}}{y_{11}}-1, \frac{w_{11}}{y_{11}}-1\Big)R_{12}
=\Big(\frac{x_5}{y_5^5}, \frac{y_5^5}{w_5}, \frac{z_5^2}{y_5^9 }-1, \frac{w_5^2}{y_5^5 z_5}-1\Big)R_{12}.
$$
Let $t_0:=0, t_1:=5$, and $t_2:=12=2t_1-t_0+2$.
For each $n \ge 3$, we let $t_n:=2t_{n-1}-t_{n-2}+3 \cdot 2^{2(n-2)}$.   
We inductively define a sequence of local quadratic transforms
with respect to the fixed  regular system of parameters $x_{t_n}, y_{t_n}, z_{t_n}, w_{t_n}$ of $\m_{t_n}$ as follows:
$$
R_{t_n}\overbrace{~\overset{y}  {\subset}~R_{t_n+1}~\overset{y}  {\subset}~\cdots ~ \overset{y}  {\subset}}^{\text{$2^{2n}$ times}}~ R_{t_n+{2^{2n}}} 
~\overset{z}  {\subset}~ R_{t_n+{2^{2n}}+1} ~\overset{w}  {\subset}~ R_{t_n+{2^{2n}}+2}~\overset{y'} {\subset}~ R_{t_n+{2^{2n}}+3}
$$
The $2^{2n}$ transforms rom $R_{t_n}$ to $R_{t_n+{2^{2n}}}$ are all monomial in the $y$-direction, 
 the transform from $R_{t_n+{2^{2n}}}$ to $R_{t_n+{2^{2n}}+1}$ is monomial in the $z$-direction, 
the transform from $R_{t_n+{2^{2n}}+1}$ to $R_{t_n+{2^{2n}}+2}$ is monomial in the $w$-direction,
and the transform from $R_{t_n+{2^{2n}}+2}$ to $R_{t_n+{2^{2n}}+3}$ is defined in a manner similar to that from $R_4$ to $R_5$.
Thus    with $\m_{t_n}=(x_{t_n}, y_{t_n}, z_{t_n}, w_{t_n})R_{t_n}$, 
we have
$$
\begin{aligned}
\m_{t_n} &\overbrace{~\overset{y}{\to}~\cdots~\overset{y}  {\to}}^{\text{${2^{2n}}$ times}}
\m_{t_n+{2^{2n}}}=\Big(\frac{x_{t_n}}{y_{t_n}^{2^{2n}}},~ y_{t_n},~ \frac{z_{t_n}}{y_{t_n}^{2^{2n}}}, \frac{w_{t_n}}{y_{t_n}^{2^{2n}}}\Big)R_{t_n+{2^{2n}}}\\
&~\overset{z}{\to}~\m_{t_n+{2^{2n}}+1}
=\Big(\frac{x_{t_n}}{z_{t_n}},~ \frac{y_{t_n}^{2^{2n}+1}}{z_{t_n}},~ \frac{z_{t_n}}{y_{t_n}^{2^{2n}}}, ~\frac{w_{t_n}}{z_{t_n}}\Big)R_{t_n+{2^{2n}}+1}\\
&~\overset{w}{\to}~\m_{t_n+{2^{2n}}+2}=\Big(\frac{x_{t_n}} {w_{t_n}},~ \frac{y_{t_n}^{2^{2n}+1}}{w_{t_n}},
~ \frac{z_{t_n}^2}{y_{t_n}^{2^{2n}}w_{t_n}},~ \frac{w_{t_n}}{z_{t_n}}\Big)R_{t_n+{2^{2n}}+2}\\
&~\overset{y'}{\to}~\m_{t_n+{2^{2n}}+3}=\Big(\frac{x_{t_n}}{y_{t_n}^{2^{2n}+1}},~ \frac{y_{t_n}^{2^{2n}+1}}{w_{t_n}},~ 
\frac{z_{t_n}^2}{y_{t_n}^{2(2^{2n})+1}}-1, \frac{w_{t_n}^2}{y_{t_n}^{2^{2n}+1}z_{t_n}}-1\Big)R_{t_n+{2^{2n}}+3}
\end{aligned}
$$
Let $S:=\bigcup_{n \ge 0}R_n$ and $\m_S:=\bigcup_{n \ge 0}\m_n$. 
Since we never divide in the $x$-direction, we have $S ~\subset~ R_{xR}$. 
Let $P:=xR_{xR} \cap S$ and let $T:=S/P$ and $\m_T:=\m_S/P$.
For each $n \ge 0$, let $x_n$ be the transform of $x$ in $R_n$.
Then $P~=~\bigcup_{n\ge 0} x_n R_n$.
For each $n \ge 0$, let
$A_n:=\frac{R_n}{x_n R_n}$ and $\n_n:=\frac{\m_n}{x_n R_n}$. 
Each $(A_n, \n_n)$ is a $3$-dimensional regular local ring with maximal ideal $\n_n$ generated 
 by images of $y_n, z_n$ and $w_n$ in $A_n$. By identifying $y_n, z_n$ and $w_n$ with their images in $A_n$,
we have $\n_n=(y_n, z_n, w_n)A_n$. 
Moreover 
$\{(A_n, \n_n)\}_{n\ge 0}$ is an infinite sequence of local quadratic transforms of $(A_0, \n_0)$, where $\n_0:=(y, z, w)A_0$.
Then $T=\bigcup_{n \ge 0}A_n$ is a normal local domain with  
 maximal ideal $\m_T=\bigcup_{n \ge 0}\n_n$.
The sequence $\{(A_n, \n_n)\}_{n\ge 0}$ is determined by the sequence $\{(R_n, \m_n)\}_{n\ge 0}$.
Thus the local quadratic transforms from $A_0$ to $A_5$ are:
$$
A_0 ~\overset{y}  {\subset}~  A_1 ~\overset{y}{\subset}~ A_2 ~\overset{z}{\subset}~ A_3~\overset{w}{\subset}~ A_4~\overset{y'}{\subset}~ A_5
$$ 
where
$$
\n_5 = (y_5, z_5, w_5)A_5 =  \Big(y_4, \frac{z_4}{y_4}-1, \frac{w_4}{y_4}-1\Big)A_5 
                          =\Big(\frac{y^3}{w},\frac{z^2}{y^5}-1, \frac{w^2}{y^3z}-1\Big)A_5
$$
Let $W$ be a valuation domain birationally  dominating $T$ and let $\omega$ be a valuation associated with $W$. 
Since $\frac{z_4}{y_4}=\frac{z^2}{y^5}$ and  $\frac{w_4}{y_4}=\frac{w^2}{y^3z}$ are in $ A_5 \setminus \n_5$, we have
$\omega(z) = \frac{5}{2} \omega(y)$ and $\omega(w)=\frac{3}{2}\omega(y)+ \frac{1}{2}\omega(z)$.
The transforms from $A_5$ to $A_{12}$ are:  
$$
A_5 ~\overset{y}{\subset}~  A_6 ~\overset{y}{\subset}~A_7 ~\overset{y}{\subset}~  A_8 ~\overset{y}{\subset}~  A_9
~\overset{z}{\subset}~  A_{10} ~\overset{w}{\subset}~ A_{11} ~\overset{y'}{\subset}~ A_{12}
$$ 
where
$$
\n_{12}=(y_{12}, z_{12}, w_{12})A_{12}=\Big(y_{11}, \frac{z_{11}}{y_{11}}-1, \frac{w_{11}}{y_{11}}-1\Big)A_{12}
=\Big(\frac{y_5^5}{w_5}, \frac{z_5^2}{y_5^9}-1, \frac{w_5^2}{y_5^5 z_5}-1\Big)A_{12}.
$$
Since $\frac{z_{12}}{y_{12}}=\frac{z_5^2}{y_5^9}$ and $\frac{w_{11}}{y_{11}}=\frac{y_5^2}{y_5^5 z_5}$ are in $ A_{12} \setminus \n_{12}$, 
we have 
$\omega (z_5) = \frac{9}{2} \omega (y_5)$ and
$\omega(w_5)=\frac{5}{2}\omega(y_5) + \frac{1}{2}\omega(z_5)$.
With the integers $t_n$ as defined in the construction of the 
sequence $\{(R_n,\m_n) \}_{n \ge 0}$, we have 
$$
A_{t_n}\overbrace{~\overset{y}  {\subset}~A_{t_n+1}~\overset{y}  {\subset}~\cdots ~ \overset{y}  {\subset}}^{\text{$2^{2n}$ times}}~ A_{t_n+{2^{2n}}} 
~\overset{z}  {\subset}~ A_{t_n+{2^{2n}}+1} ~\overset{w}  {\subset}~ A_{t_n+{2^{2n}}+2}~\overset{y'}  {\subset}~  A_{t_n+{2^{2n}}+3}.
$$
The $2^{2n}$ transforms from $A_{t_n}$ to $A_{t_n+{2^{2n}}}$ are all monomial in the $y$-direction, 
 the transform from $A_{t_n+{2^{2n}}}$ to $A_{t_n+{2^{2n}}+1}$ is monomial in the $z$-direction,
 the transform from $A_{t_n+{2^{2n}}+1}$ to $A_{t_n+{2^{2n}}+2}$ is monomial in the $w$-direction,
and with $\n_{t_n}=(y_{t_n}, z_{t_n}, w_{t_n})R_{t_n}$,  
we have 
$$
\begin{aligned}
\n_{t_n} &\overbrace{~\overset{y}{\to}~\cdots~\overset{y}  {\to}}^{\text{${2^{2n}}$ times}}
\n_{t_n+{2^{2n}}}=\Big(y_{t_n},~ \frac{z_{t_n}}{y_{t_n}^{2^{2n}}},~ \frac{w_{t_n}}{y_{t_n}^{2^{2n}}}\Big)A_{t_n+{2^{2n}}}\\
&~\overset{z}{\to}~\n_{t_n+{2^{2n}}+1}=\Big(\frac{y_{t_n}^{{2^{2n}}+1}}{z_{t_n}},~ \frac{z_{t_n}}{y_{t_n}^{{2^{2n}}}},~
 \frac{w_{t_n}}{z_{t_n}}\Big)A_{t_n+{2^{2n}}+1}\\
&~\overset{w}{\to}~\n_{t_n+{2^{2n}}+2}=\Big(\frac{ y_{t_n}^{{2^{2n}}+1} } {w_{t_n}},~ \frac{z_{t_n}^2}{y_{t_n}^{{2^{2n}}} w_{t_n}},
~ \frac{w_{t_n}}{z_{t_n}}\Big)A_{t_n+{2^{2n}}+2}\\
&~\overset{y'}{\to}~\n_{t_n+{2^{2n}}+3}=\Big(\frac{y_{t_n}^{{2^{2n}}+1}}{w_{t_n}},~ \frac{z_{t_n}^2}{y_{t_n}^{2(2^{2n})+1}}-1,~ 
\frac{w_{t_n}^2}{y_{t_n}^{2^{2n}+1} z_{t_n}}-1\Big)A_{t_n+{2^{2n}}+3}
\end{aligned}
$$
Since $\frac{z_{t_n}^2}{y_{t_n}^{2(2^{2n})+1}}$ and 
  $\frac{w_{t_n}^2}{y_{t_n}^{{2^{2n}+1}}z_{t_n}}$
in $ A_{t_n+{2^{2n}}+3} \setminus \n_{t_n+{2^{2n}}+3}$, we have
$\omega( z_{t_n})=\frac{2(2^{2n})+1}{2} \omega(y_{t_n})$ and $\omega(w_{t_n})=\frac{2^{2n}+1}{2}\omega(y_{t_n})+\frac{1}{2}\omega(z_{t_n})$.
Assume that $\omega(y)=1$. Then we have
$$
\begin{aligned}  
s:&= \sum_{i=0}^{\infty} \omega (\n_i)=1 +1+\frac{1}{2}+\Big\{\frac{1}{2^2}+\frac{1}{2^2}\Big\}
+\overbrace{\Big\{\frac{1}{2^2}+\frac{1}{2^2}+\frac{1}{2^2}+\frac{1}{2^2}\Big\}}^{\text{$2^{2\cdot1}$ times}}+\frac{1}{2^3}+\Big\{\frac{1}{2^4}+\frac{1}{2^4}\Big\}\\
&+\overbrace{\Big\{\frac{1}{2^4}+\frac{1}{2^4}+\cdots+\frac{1}{2^4}\Big\}}^{\text{$2^{2\cdot 2}$ times}}+\frac{1}{2^5}+\Big\{\frac{1}{2^6}+\frac{1}{2^6}\Big\}
+\overbrace{\Big\{\frac{1}{2^6}+\frac{1}{2^6}+\cdots+\frac{1}{2^6}\Big\}}^{\text{$2^{2\cdot 3}$ times}}+\frac{1}{2^7}
+\cdots =\infty.
\end{aligned}
$$
By Proposition~\ref{6.2}, the sequence $\{(A_n, \n_n)\}_{n \ge 0}$ switches strongly infinitely often.
Hence by  Remark~\ref{rank}, the ring $T=\bigcup_{n \ge 0}A_n$ is a valuation domain  and $T$ has rank $1$. 
Notice that $T=W$.
By Theorem~\ref{htonedirected}, the ring  $S$ is a valuation domain,  the sequence $\{(R_n, \m_n)\}_{n \ge 0}$ is height one directed,
and $S$ has rank  $2$. Let $G$ be the value group of $S$ and $H$ be the value group of $T$. 
By Remark~\ref{7.11}, the ring $S$ has rational rank $2$
and $G \cong \Z \oplus H$.
Clearly, $T$ has  rational rank $1$, and $T$ is not a DVR.
\end{example}

\end{document}